\documentclass[reqno]{amsart} 
\pdfoutput=1
\usepackage{hyperref}
\usepackage{graphicx}
\usepackage{enumerate} 
\usepackage{upref}
\usepackage{verbatim} 
\usepackage{color}
\newcommand*{\mailto}[1]{\href{mailto:#1}{\nolinkurl{#1}}}

\newtheorem{theorem}{Theorem}[section]
\newtheorem{lemma}[theorem]{Lemma}

\newtheorem{remark}[theorem]{Remark}

\newtheorem{proposition}[theorem]{Proposition}

\numberwithin{equation}{section}
\newtheorem{definition}[theorem]{Definition}
\allowdisplaybreaks

\newcommand{\act}{\bullet}

\newcommand{\Linf}{{L^\infty}}

\newcommand{\Wlocun}{{W^{1,1}_{\text{\rm loc}}}}
\newcommand{\Wperun}{{W^{1,1}_{\text{\rm per}}}}
\newcommand{\Wper}{{W^{1,\infty}_{\text{\rm per}}}}

\newcommand{\quot}{{\F/\Gr}}
\renewcommand{\H}{\mathcal{H}}

\newcommand{\epsi}{\varepsilon}
\newcommand{\Gr}{G}

\newcommand{\tnorm}[1]{\vert\hspace*{-1pt}\vert\hspace*{-1pt}\vert#1\vert\hspace*{-1pt}\vert\hspace*{-1pt}\vert}

\newcommand{\PP}{\mathcal{P}}
\newcommand{\Q}{\mathcal{Q}}
\newcommand{\Linfper}{{L^\infty_{\text{\rm per}}}}
\newcommand{\dott}{\,\cdot\,}
\newcommand{\dx}{\,dx}

\newcommand{\D}{\ensuremath{\mathcal{D}}}

\newcommand{\F}{\ensuremath{\mathcal{F}}}

\newcommand{\Real}{\mathbb{R}}
\newcommand{\N}{\mathbb{N}}

\DeclareMathOperator{\id}{id}

\newcommand{\muac}{\mu_{\text{\rm ac}}}

\DeclareMathOperator{\sign}{sign}
\DeclareMathOperator{\meas}{meas}

\newcommand{\beq}{\begin{equation}}
  \newcommand{\eeq}{\end{equation}}
\newcommand{\bal}{\begin{align}}
  \newcommand{\eal}{\end{align}}
\newcommand{\nn}{\nonumber}

\newcommand{\norm}[1]{\left\Vert#1\right\Vert}
\newcommand{\abs}[1]{\left\vert#1\right\vert}

\numberwithin{equation}{section}

\begin{document}

\title[Lipschitz metric for the periodic  Camassa--Holm
equation]{Lipschitz metric for the periodic Camassa--Holm
  equation}

\author[K. Grunert]{Katrin Grunert}
\address{Faculty of Mathematics\\ University of
  Vienna\\ Nordbergstrasse 15\\ A-1090 Wien\\
  Austria}
\email{\mailto{katrin.grunert@univie.ac.at}}
\urladdr{\url{http://www.mat.univie.ac.at/~grunert/}}

\author[H. Holden]{Helge Holden}
\address{Department of Mathematical Sciences\\
  Norwegian University of Science and Technology\\
  NO-7491 Trondheim\\ Norway\\ {\rm and} Centre of
  Mathematics for Applications\\ University of Oslo\\
  NO-0316 Oslo\\ Norway}
\email{\mailto{holden@math.ntnu.no}}
\urladdr{\url{http://www.math.ntnu.no/~holden/}}

\author[X. Raynaud]{Xavier Raynaud}
\address{Centre of Mathematics for Applications\\
  University of Oslo\\ NO-0316 Oslo\\ Norway}
\email{\mailto{xavierra@cma.uio.no}}
\urladdr{\url{http://folk.uio.no/xavierra/}}

\date{\today} 
\thanks{Research supported by the
  Research Council of Norway under Projects
  No.~195792/V11, Wavemaker, and NoPiMa.}  
  

\begin{abstract}
  We study stability of conservative solutions of the Cauchy
  problem for the periodic Camassa--Holm equation
  $u_t-u_{xxt}+3uu_x-2u_xu_{xx}-uu_{xxx}=0$ with
  initial data $u_0$.  In particular, we derive a
  new Lipschitz metric $d_\D$ with the
  property that for two solutions $u$ and $v$ of
  the equation we have
  $d_\D(u(t),v(t))\le e^{Ct}
  d_\D(u_0,v_0)$.  The relationship between this metric and usual norms in $H^1_{\rm per}$ and  $L^\infty_{\rm per}$ is clarified.  
\end{abstract}
\maketitle

\section{Introduction}

The ubiquitous Camassa--Holm (CH) equation \cite{CH:93,CHH:94}
\begin{equation} \label{eq:CH}
 u_t-u_{xxt}+\kappa u_x
 +3uu_x-2u_xu_{xx}-uu_{xxx}=0,
\end{equation}
where $\kappa\in\Real$ is a constant, has been
extensively studied due to its many intriguing
properties. The aim of this paper is to construct
a metric that renders the flow generated by the
Camassa--Holm equation Lipschitz continuous on a
function space in the conservative case. To keep
the presentation reasonably short, we restrict the
discussion to properties relevant for the current
study.

More precisely, we consider the initial value
problem for \eqref{eq:CH} with periodic initial
data $u|_{t=0}=u_0$. Since the function $v(t,x)=u(t,x-\kappa t/2)+\kappa/2$ satisfies equation 
\eqref{eq:CH} with $\kappa=0$, we can without loss of generality assume that $\kappa$ vanishes.  For convenience we assume
that the period is $1$, that is, $u_0(x+1)=u_0(x)$
for $x\in\Real$.  The natural norm for this
problem is the usual norm in the Sobolev space
$H^1_{\rm per}$ as we have that
\begin{equation} \label{eq:H1}
\frac{d}{dt}\norm{u(t)}^2_{H^1_{\rm per}}= 
\frac{d}{dt}\int_0^1\big(u^2+u_x^2 \big)dx=2\int_0^1\big(uu_t+u_x u_{xt} \big)dx
=0
\end{equation}
(by using the equation and several integration by
parts as well as periodicity) for smooth solutions
$u$. Even for smooth initial data, the solutions
may develop singularities in finite time and this breakdown of
solutions is referred to as wave breaking.  At wave breaking the $H^1$ and
$L^\infty$ norms of the solution remain finite
while the spatial derivative $u_x$ becomes
unbounded pointwise. This phenomenon can best be
described for a particular class of solutions,
namely the multipeakons.  For simplicity we
describe them on the full line, but similar
results can be described in the periodic
case. Multipeakons are solutions of the form (see
also \cite{HolRay:06b})
\begin{equation}
\label{eq:chP}
u(t,x)=\sum_{i=1}^{n} p_i(t)e^{-\abs{x-q_i(t)}}.
\end{equation}
Let us consider the case with $n=2$ and one peakon
$p_1(0)>0$ (moving to the right) and one
antipeakon $p_2(0)<0$ (moving to the left). In the
symmetric case ($p_1(0)=-p_2(0)$ and
$q_1(0)=-q_2(0)<0$) the solution $u$ will vanish
pointwise at the collision time $t^*$ when
$q_1(t^*)=q_2(t^*)$, that is, $u(t^*,x)=0$ for all
$x\in \Real$. Clearly the well-posedness, in
particular, Lipschitz continuity, of the solution
is a delicate matter. Consider, e.g., the
multipeakon $u^\epsi$ defined as
$u^\epsi(t,x)=u(t-\epsi,x)$, see Figure
\ref{fig:peakcol}.  For simplicity, we assume that
$\norm{u(0)}_{H^1}=1$. Then, we have
\begin{equation*}
\text{$\lim_{\epsi\to0}\norm{u(0)-u^\epsi(0)}_{H^1}=0$ and 
$\norm{u(t^*)-u^\epsi(t^*)}_{H^1}=\norm{u^\epsi(t^*)}_{H^1}=1$},
\end{equation*}
and the flow is clearly not Lipschitz continuous
with respect to the $H^1$ norm.
\begin{figure}
  \includegraphics[width=6.8cm]{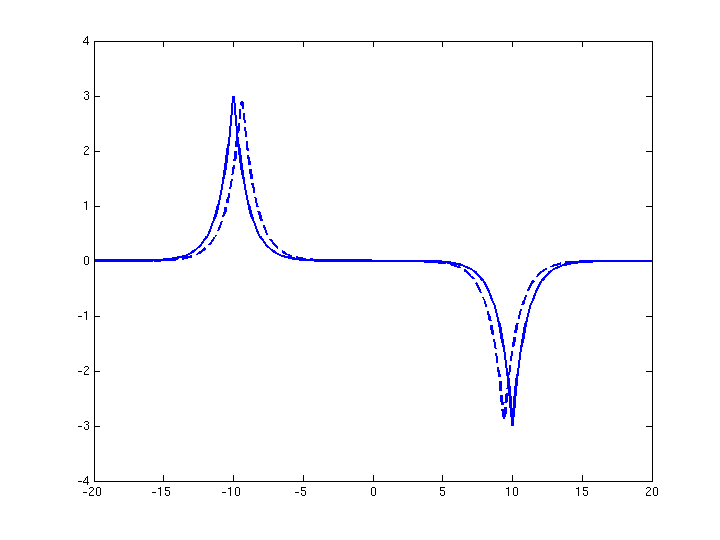}
  \includegraphics[width=6.8cm]{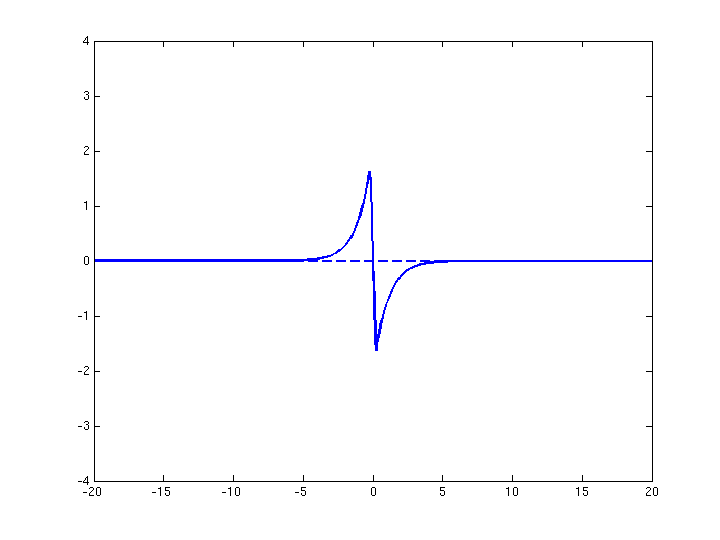}
  \label{fig:peakcol}
  \caption{The dashed curve depicts the antisymmetric multipeakon solution
    $u(t,x)$, which vanishes at $t^*$,  for $t=0$ (on the left) and
    $t=t^*$ (on the right). The solid curve depicts 
     the multipeakon solution given by
    $u^\epsi(t,x)=u(t-\epsi,x)$.}
\end{figure}

Our task is here to identify a metric, which we
will denote by $d_\D$ for which conservative
solutions satisfy a Lipschitz property, that is,
if $u$ and $v$ are two solutions of the
Camassa--Holm equation, then
\begin{equation*}
  d_\D(u(t),v(t))\le C_T d_\D(u_0,v_0), \quad t\in [0,T]
\end{equation*}
for any given, positive $T$. For nonlinear partial
differential equations this is in general a quite
nontrivial issue. Let us illustrate it in the case
of hyperbolic conservation laws
\begin{equation*}
  u_t+f(u)_x=0, \quad u|_{t=0}=u_0.
\end{equation*} 
In the scalar case with $u=u(x,t)\in\Real$,
$x\in\Real$, it is well-known \cite{HR} that the
solution is $L^1$-contractive in the sense that
\begin{equation*}
  \norm{u(t)-v(t)}_{L^1(\Real)}\le \norm{u_0-v_0}_{L^1(\Real)}, \quad t\in[0,\infty).
\end{equation*} 
In the case of systems, i.e., for $u\in\Real^n$
with $n>1$ it is known \cite{HR} that
\begin{equation*}
  \norm{u(t)-v(t)}_{L^1(\Real)}\le C\norm{u_0-v_0}_{L^1(\Real)}, \quad t\in[0,\infty),
\end{equation*} 
for some constant $C$. More relevant for the
current study, but less well-known, is the recent
analysis \cite{BHR} of the Hunter--Saxton (HS)
equation
\begin{equation}
  \label{eq:hs0}
  u_t+uu_x=\frac14\Big(\int_{-\infty}^xu_x^2\,dx-\int_{x}^\infty
  u_x^2\,dx\Big), \quad u|_{t=0}=u_0,
\end{equation}
or alternatively
\begin{equation}
  (u_t+uu_x)_x=\frac12 u_x^2, \quad u|_{t=0}=u_0,
\end{equation}
which was first introduced in \cite{MR1135995} as
a model for liquid crystals.  Again the equation
enjoys wave breaking in finite time and the
solutions are not Lipschitz in term of convex
norms. The Hunter--Saxton equation can in some sense
be considered as a simplified version of the
Camassa--Holm equation, and the construction of the
semigroup of solutions via a change of coordinates
given in \cite{BHR} is very similar to
the one used here and in
\cite{HolRay:07a} for the Camassa--Holm
equation. In \cite{BHR} the authors constructed a
Riemannian metric which renders the conservative flow generated
by the Hunter--Saxton equation Lipschitz
continuous on an appropriate function space.

For the Camassa--Holm equation, the problem of
continuation beyond wave breaking has been
considered by Bressan and Constantin
\cite{MR2278406,BreCons:05a} and Holden and
Raynaud \cite{HolRay:07,HolRay:07a,HolRay:09} (see
also Xin and Zhang \cite{xin_zhang,xin_zhang1} and
Coclite, Karlsen, and Holden \cite{CHK1,CHK2}).
Both approaches are based on a reformulation
(distinct in the two approaches) of the
Camassa--Holm equation as a semilinear system of
ordinary differential equations taking values in a
Banach space. This formulation allows one to
continue the solution beyond collision time,
giving either a global conservative solution where
the energy is conserved for almost all times or a
dissipative solution where energy may vanish from
the system. Local existence of the semilinear
system is obtained by a contraction argument.
Going back to the original function $u$, one
obtains a global solution of the Camassa--Holm
equation.

In \cite{BreFon:05}, Bressan and Fonte introduce a
new distance function $J(u,v)$ which is defined as
a solution of an optimal transport problem. They
consider two multipeakon solutions $u(t)$ and
$v(t)$ of the Camassa--Holm equation and prove, on the intervals of times where no collisions
occur, that the growth of $J(u(t),v(t))$ is linear
(that is, $\frac{dJ}{dt}(u(t),v(t))\leq
CJ(u(t),v(t))$ for some fixed constant $C$) and
that  $J(u(t),v(t))$ is
continuous across collisions. It follows that
\begin{equation}
  \label{eq:LipJ}
  J(u(t),v(t))\leq e^{CT}J(u(0),v(0))
\end{equation}
for all times $t$ that are not collision times
and, in particular, for almost all times. By
density, they construct solutions for any initial
data (not just the multipeakons) and the Lipschitz
continuity follows from \eqref{eq:LipJ}. As in
\cite{BreFon:05}, the goal of this article is to
construct a metric which makes the flow Lipschitz
continuous. However, we base the construction of
the metric directly on the reformulation of the
equation which is used to construct the solutions
themselves, and we use some fundamental geometrical
properties of this reformulation (relabeling
invariance, see below). The metric is defined on
the set $\D$ which includes configurations where
part of the energy is concentrated on sets of
measure zero; a natural choice for conservative solutions. In particular, we
obtain that the Lipschitz continuity holds for all
times and not just for almost all times as in
\cite{BreFon:05}.

Let us describe in some detail the approach in
this paper, which follows \cite{HolRay:07a} quite
closely in setting up the reformulated equation.
Let $u=u(t,x)$ denote the solution, and $y(t,\xi)$
the corresponding characteristics, thus
$y_t(t,\xi)=u(t,y(t,\xi))$. Our new variables are
$y(t,\xi)$,
\begin{equation}
  U(t,\xi)=u(t,y(t,\xi)), \quad
  H(t,\xi)=\int_{y(t,0)}^{y(t,\xi)}(u^2+u_x^2)\dx
\end{equation}
where $U$ corresponds to the Lagrangian velocity
while $H$ could be interpreted as the Lagrangian
cumulative energy distribution. In the periodic case one defines
\begin{align}
  \label{eq:Qsimp1}
  Q&=\frac1{2(e-1)}\int_0^1\sinh(y(\xi)-y(\eta))(U^2y_\xi+H_\xi)(\eta)\,d\eta \\
  &\qquad -\frac14\int_0^1\sign(\xi-\eta)\exp\big(-\sign(\xi-\eta)(y(\xi)-y(\eta))\big)(U^2y_\xi+H_\xi)(\eta)\,d\eta,\notag \\
  \label{eq:Psimp1}
  P&=\frac1{2(e-1)}\int_0^1\cosh(y(\xi)-y(\eta))(U^2y_\xi+H_\xi)(\eta)\,d\eta\\
  &\qquad+\frac14\int_0^1\exp\big(-\sign(\xi-\eta)(y(\xi)-y(\eta))\big)(U^2y_\xi+H_\xi)(\eta)\,d\eta.\notag
\end{align}
Then one can show that 
\begin{equation}
  \left\{
    \begin{aligned}
      \label{eq:sysA}
      y_t&=U,\\
      U_t&=-Q,\\
      H_t&=[U^3-2PU]_0^\xi,
    \end{aligned}
  \right.
\end{equation}
is equivalent to the Camassa--Holm equation. 
Global existence of solutions of \eqref{eq:sysA}
is obtained starting from a contraction argument,
see Theorem \ref{th:global}.  The issue of continuation of the solution past wave breaking
is  resolved by considering
the set $\D$ (see Definition \ref{def:D}) which
consists of pairs $(u,\mu)$ such that
$(u,\mu)\in\D$ if $u\in H^1_{\text{\rm per}}$ and
$\mu$ is a positive Radon measure with period one,
and whose absolutely continuous part satisfies
$\muac=(u^2+u_x^2)\,dx$. With three Lagrangian
variables $(y,U,H)$ versus two Eulerian variables
$(u,\mu)$, it is clear that there can be no
bijection between the two coordinate systems.  If
two Lagrangian variables correspond to one and the
same solution in Eulerian variables, we say that
the Lagrangian variables are relabelings of each
other.  To resolve the relabeling issue we define
a group of transformations which acts on the
Lagrangian variables and lets the system of
equations \eqref{eq:sysA} invariant. We are able
to establish a bijection between the space of
Eulerian variables and the space of Lagrangian
variables when we identify variables that are
invariant under the action of the group. This
bijection allows us to transform the results
obtained in the Lagrangian framework (in which the
equation is well-posed) into the Eulerian
framework (in which the situation is much more
subtle).  To obtain a Lipschitz metric in Eulerian
coordinates we start by constructing one in the
Lagrangian setting. To this end we start by
identifying a set $\F$ (see Definition
\ref{def:F}) that leaves the flow \eqref{eq:sysA}
invariant, that is, if $X_0\in \F$ then the
solution $X(t)$ of \eqref{eq:sysA} with $X(0)=X_0$
will remain in $\F$, i.e., $X(t)\in \F$. Next, we
identify a subgroup $\Gr$, see Definition
\ref{def:G}, of the group of homeomorphisms on the
unit interval, and we interpret $\Gr$ as the set
of relabeling functions. From this we define a
natural group action of $\Gr$ on $\F$, that is,
$\Phi(f,X)=X\act f$ for $f\in \Gr$ and $X\in\F$,
see Definition \ref{def:group} and Proposition
\ref{prop:action}.  We can then consider the
quotient space $\quot$. However, we still have to
identify a unique element in $\F$ for each
equivalence class in $\quot$. To this end we
introduce the set $\H$, see \eqref{eq:calH}, of
elements in $\F$ for which $\int_0^1 y(\xi)d\xi=0$
and $y_\xi+H_\xi=1+\norm{H_\xi}_{L^1}$.  This
establishes a bijection between $\quot$ and $\H$,
see Lemma \ref{lemma:bijection}, and therefore
between $\H$ and $\D$.  Finally, we define a
semigroup $\bar S_t(X_0)=X(t)$ on $\H$ (Definition
\ref{def:barS}), and the next task is to identify
a metric that makes the flow $\bar S_t$ Lipschitz
continuous on $\H$. We use the bijection between
$\H$ and $\D$ to transport the metric from $\H$ to
$\D$ and get a Lipschitz continuous flow on $\D$.

In \cite{HolRay:07a}, the authors define the
metric on $\H$ by simply taking the norm of the
underlying Banach space (the set $\H$ is a
nonlinear subset of a Banach space). They obtain
in this way a metric which makes the flow
continuous but not Lipschitz continuous. As we
will see (see Remark \ref{rem:compmetric}), this
metric is stronger than the one we construct here
and for which the flow is Lipschitz continuous. In
\cite{BHR}, for the Hunter--Saxton equation, the
authors use ideas from Riemannian geometry and
construct a semimetric which identifies points
that belong to the same equivalence class. The
Riemannian framework seems however too rigid for
the Camassa--Holm equation, and we have not been
able to carry out this approach. However, we
retain the essential idea which consists of
finding a semimetric which identifies equivalence
classes. Instead of a Riemannian metric, we use a
discrete counterpart. Note that this technique
will also work for the Hunter--Saxton and will
give the same metric as in \cite{BHR}. A natural
candidate for a semimetric which identifies
equivalence classes is (cf.~\eqref{eq:defJ})
\begin{equation*}
  J(X,Y)=\inf_{f,g\in\Gr}\norm{X\act f-Y\act g},
\end{equation*}
which is invariant with respect to relabeling. However, it does not satisfy the triangle inequality. Nevertheless it can be modified to satisfy all the requirements for a metric if we instead define, see Definition \ref{def:alberto}, the following quantity\footnote{This idea is due to 
  A.~Bressan (private communication).}
\begin{equation}\label{eq:alberto1}
  d(X,Y)=\inf \sum_{i=1}^NJ(X_{n-1},X_n)
\end{equation}
where the infimum is taken over all finite sequences
$\{X_n\}_{n=0}^N\in\F$ which satisfy
$X_0=X$ and $X_N=Y$.  One can then prove that $d(X,Y)$ is a metric on $\H$, see 
Lemma \ref{lemma:distance}.  Finally, we prove that the flow is Lipschitz continuous in this metric, see Theorem \ref{th:stab}.  To transfer this result to the Eulerian variables we 
reconstruct these variables from the Lagrangian coordinates as in \cite{HolRay:07a}:  
Given $X\in \F$, we define $(u,\mu)\in\D$ by (see Definition \ref{def:FtoE})  $u(x)=U(\xi)$ for 
any   $\xi$ such that $x=y(\xi)$, and $\mu =y_\# (\nu d\xi)$. We denote the mapping from $\F$ to $\D$ by $M$, and the inverse restricted to $\H$ by $L$.  The natural metric on $\D$, denoted $d_\D$, is then defined by  $d_{\D}((u,\mu),(\tilde{u},\tilde{\mu}))=d(L(u,\mu),L(\tilde{u},\tilde{\mu}))$ for two elements $(u,\mu),(\tilde{u},\tilde{\mu})$ in $\D$, see Definition \ref{def:dD}.  
The main theorem, Theorem \ref{th:main}, then states that the metric $d_\D$ is Lipschitz continuous on all states with finite energy.   In the last section, Section \ref{sec:topology}, the metric is compared with the standard norms.   Two results are proved: The mapping 
$u\mapsto  (u,(u^2+u_x^2) dx)$
is continuous from $H^1_{\rm per}$ into
$\D$ (Proposition \ref{prop:cont1}).  Furthermore, if  $(u_n, \mu_n)$ is a sequence in
$\D$ that converges to $(u,\mu)$ in
$\D$. Then $u_n\rightarrow u$ in $L^\infty_{\rm per}$  and  $\mu_n \overset{\ast}{\rightharpoonup}\mu$ (Proposition \ref{prop:cont2}).

The problem of Lipschitz continuity can nicely be illustrated in
the simpler context of ordinary differential
equations.  Consider three differential equations:
\begin{subequations}
  \begin{align}
    \dot x&= a(x), & x(0)&=x_0, \quad \text{$a$ Lipschitz},  \label{eq:reg}\\
    \dot x&= 1+\alpha H(x),&  x(0)&=x_0, \quad \text{$H$ the Heaviside function,
      $\alpha>0$},
    \label{eq:Lip}\\
    \dot x&=\abs{x}^{1/2}, & x(0)&=x_0, \quad\quad
    \text{$t\mapsto x(t)$ strictly increasing}. \label{eq:sqrtlaw}
  \end{align}
\end{subequations}
Straightforward computations give as solutions
\begin{subequations}
  \begin{align}
    x(t)&=x_0+\int_0^t a(x(s))\, ds,  \\
    x(t)&=(1+\alpha H(t-t_0))(t-t_0), \quad t_0=-x_0/(1+\alpha H(x_0)), \\
    x(t)&=\sign\big(\frac{t}2+v_0\big) \big(\frac{t}{2}+v_0\big)^2\text{ where }v_0
    =\sign(x_0)\abs{x_0}^{1/2}.
  \end{align}
\end{subequations}
We find that
\begin{subequations}
  \begin{align}
    \abs{x(t)-\bar x(t)}&\le e^{Lt}\abs{x_0-\bar x_0}, \quad L= \norm{a}_{\rm Lip},  \\
    \abs{x(t)-\bar x(t)}&\le (1+ \alpha) \abs{x_0-\bar x_0}, \\
    x(t)-\bar x(t)&= t(x_0-\bar
    x_0)^{1/2}+\abs{x_0-\bar x_0}, \quad
    \text{when $\bar x_0=0$, $t>0$, $x_0>0$.}
  \end{align}
\end{subequations}
Thus we see that in the regular case
\eqref{eq:reg} we get a Lipschitz estimate with
constant $e^{Lt}$ uniformly bounded as $t$ ranges
on a bounded interval.  In the second case
\eqref{eq:Lip} we get a Lipschitz estimate
uniformly valid for all $t\in \Real$.  In the
final example \eqref{eq:sqrtlaw}, by restricting
attention to strictly increasing solutions of the
ordinary differential equations, we achieve uniqueness and continuous
dependence on the initial data, but without any
Lipschitz estimate at all near the point $x_0=0$.  We observe that, by
introducing the Riemannian metric
\begin{equation}
  \label{eq:defriemmet}
  d(x, \bar x)= \abs{\int_x^{\bar x} \frac{dz}{{\abs{z}^{1/2}}}},
\end{equation}
an easy computation reveals that
\begin{equation}
  d(x(t), \bar x(t))= d(x_0, \bar x_0).
\end{equation}
Let us explain why this metric can be considered
as a Riemannian metric.
The Euclidean metric between the two points is
then given
\begin{equation}
  \label{eq:eucldist}
  \abs{x_0-\bar x_0}=\inf_{x}\int_0^1\abs{x_s(s)}\,ds
\end{equation}
where the infimum is taken over all paths
$x\colon[0,1]\to\Real$ that join the two points $x_0$
and $\bar x_0$, that is, $x(0)=x_0$ and
$x(1)=\bar x_0$. However, as we have seen, the
solutions are not Lipschitz for the Euclidean
metric. Thus we want to measure the infinitesimal
variation $x_s$ in an alternative way, which makes
solutions of equation \eqref{eq:sqrtlaw} Lipschitz
continuous. We look at the evolution equation that
governs $x_s$ and, by differentiating
\eqref{eq:sqrtlaw} with respect to $s$, we get
\begin{equation*}
  \dot x_s=\frac{\sign(x)x_s}{2\sqrt{\abs{x}}},
\end{equation*}
and we can check that
\begin{equation}
  \label{eq:dtriemevol}
  \frac{d}{dt}\left(\frac{\abs{x_s}}{\sqrt{\abs{x}}}\right)=0.
\end{equation}
Let us consider the real line as a Riemannian
manifold where, at any point $x\in\Real$, the
Riemannian norm is given
by $\abs{v}/\sqrt{\abs{x}}$ for any tangent vector
$v\in\Real$ in the tangent space of $x$.  From
\eqref{eq:dtriemevol}, one can see that at the
infinitesimal level, this Riemannian norm is
exactly preserved by the evolution equation. The
distance on the real line which is naturally
inherited by this Riemannian is given by
\begin{equation*}
  d(x_0,\bar x_0)=\inf_{x}\int_0^1\frac{\abs{x_s}}{\sqrt{\abs{x}}}\,ds
\end{equation*}
where the infimum is taken over all paths
$x\colon[0,1]\to\Real$ joining $x_0$ and $\bar
x_0$. It is quite reasonable to restrict ourselves
to paths that satisfy $x_s\geq0$ and then, by a
change of variables, we recover the definition
\eqref{eq:defriemmet}.  

The Riemannian approach to measure a distance
between any two distinct points in a given set (as
defined in \eqref{eq:eucldist}) requires the
existence of a smooth path between points in the
set. In the case of the Hunter--Saxton (see
\cite{BHR}), we could embed the set we were
primarily interested in into a convex set (which
is therefore connected) and which also could be
regularized (so that the Riemannian metric we
wanted to use in that case could be defined). In
the case of the Camassa--Holm equation, we have
been unable to construct such a set. However, there
exists the alternative approach which, instead of
using a smooth path to join points, uses finite
sequences of points, see \eqref{eq:alberto1}. We
illustrate this approach with equation
\eqref{eq:sqrtlaw}. We want to define a metric in
$(0,\infty)$ which makes the semigroup of
solutions Lipschitz stable. Given two points
$x,\bar{x}\in (0,\infty)$, we define the function
$J\colon (0,\infty)\times (0,\infty)\to[0,\infty)$
as
\begin{equation*}
  J(x,\bar{x})=
  \begin{cases}
    \frac{x-\bar{x}}{\bar{x}^{1/2}}&\text{
      if $x\geq \bar{x}$,}\\[2mm]
    \frac{\bar{x}-x}{x^{1/2}}&\text{
      if $x<\bar{x}$}.
  \end{cases}
\end{equation*}
The function $J$ is symmetric and
$J(x,\bar{x})=0$ if and only if
$x=\bar{x}$, but $J$ does not satisfy the
triangle inequality. Therefore we define (cf.~\eqref{eq:alberto1})
\begin{equation} \label{eq:alberto}
  d(x,\bar{x})=\inf \sum_{n=0}^NJ(x_n,x_{n+1})
\end{equation}
where the infimum is taken over all finite sequences
$\{x_n\}_{n=0}^N$ such that $x_0=x$ and
$x_N=\bar{x}$. Then, $d$ satisfies the triangle
inequality and one can prove that it is also a
metric. Given $x_n,x_{n+1}\in E$ such that
$x_n\leq x_{n+1}$, we denote $x_n(t)$ and
$x_{n+1}(t)$ the solution of \eqref{eq:sqrtlaw}
with initial data $x_n$ and $x_{n+1}$,
respectively. After a short computation, we get
\begin{equation*}
  \frac{d}{dt}J(x_n(t),x_{n+1}(t))=-\frac1{2x_n}(x_n+x_{n+1}-2\sqrt{x_nx_{n+1}})\leq0.
\end{equation*}
Hence, $J(x_n(t),x_{n+1}(t))\leq J(x_n,x_{n+1})$
so that
\begin{equation*}
  d(x(t),\bar{x}(t))\leq d(x,\bar{x})
\end{equation*}
and the semigroup of solutions to
\eqref{eq:sqrtlaw} is a contraction for the metric
$d$. It follows from the definition of $J$ that,
for $x_1,x_2,x_3\in E$ with $x_1<x_2<x_3$, we have
\begin{equation}
  \label{eq:stricineqj}
  J(x_1,x_2)+J(x_2,x_3)<J(x_1,x_3).
\end{equation}
It implies that $d(x,\bar x)$ satisfies
\begin{equation*}
  d(x,\bar x)=\inf_{\delta} \sum_{n=0}^NJ(x_n,x_{n+1})
\end{equation*}
where $\delta=\min_{n}\abs{x_{n+1}-x_n}$, which is
also the definition of the Riemann integral, so
that
\begin{equation*}
  d(x,\bar x)=\int_{x}^{\bar x}\frac{1}{\sqrt{z}}\,dz
\end{equation*}
and the metric we have just defined coincides with
the Riemannian metric we have introduced. Note
that if we choose
\begin{equation*}
  \bar J(x,\bar{x})=
  \begin{cases}
    \frac{x-\bar{x}}{x^{1/2}}&\text{
      if $x\geq\bar{x}$}\\
    \frac{\bar{x}-x}{\bar{x}^{1/2}}&\text{ if
      $x<\bar{x}$},
  \end{cases}
\end{equation*}
then \eqref{eq:stricineqj} does not hold; we have
instead $\bar J(x_1,x_3)<\bar J(x_1,x_2)+\bar J(x_2,x_3)$, which is the triangle inequality. Thus,
for $\bar d$ as defined by \eqref{eq:alberto} with $J$ replaced by $\bar J$, we get
\begin{equation*}
  \bar d(x,\bar x)=\bar J(x,\bar x)\neq\int_{x}^{\bar x}\frac{1}{\sqrt{z}}\,dz.
\end{equation*}
It is also possible to check that, for $\bar J$,
we cannot get that $\bar
J(x_n(t),x_{n+1}(t))\leq C \bar J(x_n,x_{n+1})$ for any constant $C$ 
for any $x_n$ and $x_{n+1}$ and $t\in[0,T]$ (for a
given $T$), so that the definition of $\bar J$ is
inappropriate to obtain results of stability for
\eqref{eq:sqrtlaw}.

\section{Semi-group of solutions in Lagrangian
  coordinates}

The Camassa--Holm equation for $\kappa=0$
reads
\begin{equation}
  u_t-u_{xxt}+3uu_x-2u_xu_{xx}-uu_{xxx}=0,
\end{equation}
and can be rewritten as the following 
system\footnote{For $\kappa$ nonzero, equation
\eqref{eq:Pa} is simply replaced by
$P-P_{xx}=\kappa u+u^2+\frac{1}{2}u_x^2$.}
\begin{align}
  u_t+uu_x+P_x&=0, \label{eq:Pa} \\
  P-P_{xx}&=u^2+\frac{1}{2}u_x^2.
\end{align}

We consider periodic solutions of period one.
Next, we rewrite the equation in Lagrangian
coordinates. Therefore we introduce the
characteristics
\begin{equation}
  y_t(t,\xi)=u(t,y(t,\xi)).
\end{equation}
We introduce the space $V_1$ defined as
\begin{equation*}
  V_1=\{f\in \Wlocun(\Real) \mid  f(\xi+1)=f(\xi)+1\, \text{for all }\xi\in\Real\}.
\end{equation*}
Functions in $V_1$ map the unit interval into
itself in the sense that if $u$ is periodic with
period 1, then $u\circ f$ is also periodic with
period 1. The Lagrangian velocity $U$ reads
\begin{equation}
  U(t,\xi)=u(t,y(t,\xi)).
\end{equation}
We will consider $y\in V_1$ and $U$ periodic.  We
define the Lagrangian energy cumulative
distribution as
\begin{equation}
  \label{eq:Hdef}
  H(t,\xi)=\int_{y(t,0)}^{y(t,\xi)}(u^2+u_x^2)(t,x)\,dx.
\end{equation}
For all $t$, the function $H$
belongs to the vector space $V$ defined as
follows:
\begin{multline*}
  V=\{f\in \Wlocun(\Real)\mid  \text{there exists $\alpha\in\Real$}\\ 
  \text{such that $f(\xi+1)=f(\xi)+\alpha$ for all $\xi\in\Real$}\}.
\end{multline*}
Equip $V$ with the norm
\begin{equation*}
  \norm{f}_{V}=\norm{f}_{L^\infty([0,1])}+\norm{f_\xi}_{L^1([0,1])}. 
\end{equation*}
As an immediate consequence of the definition of
the characteristics we obtain
\begin{equation}
  U_t(t,\xi)=u_t(t,y)+y_t(t,\xi)u_x(t,y)=-P_x\circ y(t,\xi).
\end{equation}
This last term can be expressed uniquely in term
of $U$, $y$, and $H$. We have the following
explicit expression for $P$,
\begin{equation}
  \label{eq:peq}
  P(t,x)=\frac{1}{2}\int_\Real
  e^{-\abs{x-z}}(u^2(t,z)+\frac{1}{2}u_x^2(t,z))\,dz.
\end{equation} 
Thus,
\begin{equation*}
  P_x\circ{y}(t,\xi)=-\frac{1}{2}\int_\Real\sign(y(t,\xi)-z) e^{-\abs{y(t,\xi)-z}}(u^2(t,z)+\frac{1}{2}u_x^2(t,z))\,dz,
\end{equation*}
and, after the change of variables $z=y(t,\eta)$,
\begin{multline}
  \label{eq:pxy0}
  P_x\circ y(t,\xi)=-\frac{1}{2}\int_\Real\Big[\sign(y(t,\xi)-y(t,\eta)) e^{-\abs{y(t,\xi)-y(t,\eta)}}\\
  \times\left(u^2(t,y(t,\eta))+\frac{1}{2}u_x^2(t,y(t,\eta))\right)y_\xi(t,\eta)\Big]\,d\eta.
\end{multline}
We have
\begin{equation}
  \label{eq:Hxi}
  H_\xi=(u^2+u_x^2)\circ{y}y_\xi=: \nu.
\end{equation}
Note that $\nu$ is
periodic with period one. Then, \eqref{eq:pxy0}
can be rewritten as
\begin{equation}
  \label{eq:pxy1}
  P_x\circ{y}(\xi)=-\frac{1}{4}\int_\Real\sign(y(\xi)-y(\eta))\exp(-\abs{y(\xi)-y(\eta)})\big(U^2y_\xi+\nu\big)(\eta)\,d\eta,
\end{equation}
where the $t$ variable has been dropped to
simplify the notation. Later we will prove that
$y$ is an increasing function for any fixed time
$t$. If, for the moment, we take this for granted,
then $P_x\circ{y}$ is equivalent to $Q$ where
\begin{equation}
  \label{eq:Q}
  Q(t,\xi)=-\frac{1}{4}\int_\Real\sign(\xi-\eta)\exp\big(-\sign(\xi-\eta)(y(\xi)-y(\eta))\big)\big(U^2y_\xi+\nu\big)(\eta)\,d\eta,
\end{equation}
and, slightly abusing the notation, we write
\begin{equation}
  \label{eq:P}
  P(t,\xi)=\frac{1}{4}\int_\Real\exp\big(-\sign(\xi-\eta)(y(\xi)-y(\eta))\big)\big(U^2y_\xi+\nu\big)(\eta)\,d\eta.
\end{equation}
The derivatives of $Q$ and $P$ are given by
\begin{equation}
  \label{eq:QPder}
  Q_\xi=-\frac{1}{2}\nu-\left(\frac{1}{2}U^2-P\right)y_\xi\ \text{ and }\ P_\xi=Qy_\xi,
\end{equation}
respectively. 
For $\xi\in[0,1]$, using the fact that
$y(\xi+1)=y(\xi)+1$ and the periodicity of $\nu$
and $U$, the expressions  for $Q$ and $P$ can be rewritten as
\begin{multline}
  \label{eq:Qsimp}
  Q=\frac1{2(e-1)}\int_0^1\sinh(y(\xi)-y(\eta))(U^2y_\xi+\nu)(\eta)\,d\eta\\-\frac14\int_0^1\sign(\xi-\eta)\exp\big(-\sign(\xi-\eta)(y(\xi)-y(\eta))\big)(U^2y_\xi+\nu)(\eta)\,d\eta,
\end{multline}
and
\begin{multline}
  \label{eq:Psimp}
  P=\frac1{2(e-1)}\int_0^1\cosh(y(\xi)-y(\eta))(U^2y_\xi+\nu)(\eta)\,d\eta\\+\frac14\int_0^1\exp\big(-\sign(\xi-\eta)(y(\xi)-y(\eta))\big)(U^2y_\xi+\nu)(\eta)\,d\eta.
\end{multline}
Thus $P_x\circ y$ and $P\circ y$ can be replaced
by equivalent expressions given by \eqref{eq:Q}
and \eqref{eq:P} which only depend on our new
variables $U$, $H$, and $y$. We obtain a new
system of equations, which is at least formally
equivalent to the Camassa--Holm equation:
\begin{equation}
  \label{eq:equivsys}
  \left\{
    \begin{aligned}  
      y_t &= U,\\
      U_t &= -Q,\\
      H_t &= [U^3-2PU]_0^\xi.
    \end{aligned}
  \right.
\end{equation}
After differentiating  \eqref{eq:equivsys} we find
\begin{equation}
  \label{eq:equivsysdev}
  \left\{
    \begin{aligned}
      y_{\xi t}& =U_\xi,\\
      U_{\xi t}&=\frac{1}{2}\nu+\left(\frac{1}{2}U^2-P\right)y_\xi,\\
      H_{\xi
        t}&=-2Q\,Uy_\xi+\left(3U^2-2P\right)U_\xi.
    \end{aligned}
  \right.
\end{equation}
From \eqref{eq:equivsys} and
\eqref{eq:equivsysdev}, we obtain the system
\begin{equation}
  \label{eq:sys}
  \left\{
    \begin{aligned}  
      y_t &= U,\\
      U_t &= -Q,\\
      \nu_t&=-2Q\,Uy_\xi+\left(3U^2-2P\right)U_\xi.
    \end{aligned}
  \right.
\end{equation}
We can write   \eqref{eq:sys} more compactly as
\begin{equation}
  \label{eq:condequivsys}
  X_t=F(X), \quad X=(y,U,\nu).
\end{equation}
Let
\begin{equation*}
  \Wperun=\{f\in \Wlocun(\Real) \mid  f(\xi+1)=f(\xi)\text{ for all }\xi\in\Real\}.
\end{equation*}
We equip $\Wperun$ with the norm of $V$, that is,
\begin{equation*}
  \norm{f}_{\Wperun}=\norm{f}_{L^\infty([0,1])}+\norm{f_\xi}_{L^1([0,1])},
\end{equation*}
which is equivalent to the standard norm of
$\Wperun$ because
$\norm{f}_{L^1([0,1])}\leq\norm{f}_{L^\infty([0,1])}\leq\norm{f}_{L^1([0,1])}+\norm{f_\xi}_{L^1([0,1])}$.
Let
$E$ be the Banach space defined as
\begin{equation*}
  E=V\times \Wperun\times L^1_{\rm per}.
\end{equation*}
We derive the following Lipschitz estimates for
$P$ and $Q$.
\begin{lemma}
  \label{lem:PQ} 
  For any $X=(y,U,\nu)$ in $E$, we define the maps
  $\Q$ and $\PP$ as $\Q(X)=Q$ and $\PP(X)=P$ where
  $Q$ and $P$ are given by \eqref{eq:Q} and
  \eqref{eq:P}, respectively. Then, $\PP$ and $\Q$
  are Lipschitz maps on bounded sets from $E$ to
  $\Wperun$. More precisely, we have the following
  bounds. Let 
  \begin{equation}
    \label{eq:defBM}
    B_M=\{X=(y,U,\nu)\in E\mid 
    \norm{U}_{\Wperun}+\norm{y_\xi}_{L^1}+\norm{\nu}_{L^1}\leq
    M\}.
  \end{equation}
  Then for any $X,\tilde X\in B_M$, we have
  \begin{equation}
    \label{eq:lipbdQ}
    \norm{\Q(X)-\Q(\tilde X)}_{\Wperun}\leq C_M\norm{X-\tilde X}_{E}
  \end{equation}
  and
  \begin{equation}
    \label{eq:lipbdP}
    \norm{\PP(X)-\PP(\tilde X)}_{\Wperun}\leq C_M\norm{X-\tilde X}_{E}
  \end{equation}
  where the constant $C_M$ only depends on the value
  of $M$.
\end{lemma}

\begin{proof} 
  Let us first prove that $\PP$ and $\Q$ are
  Lipschitz maps from $B_M$ to $\Linfper$. Note that by using a change of variables in 
  \eqref{eq:Qsimp} and \eqref{eq:Psimp}, we obtain that $\PP$ and $\Q$ are periodic with period $1$.  Let now
  $X=(y,U,\nu)$ and $\tilde X=(\tilde y,\tilde
  U,\tilde \nu)$ be two elements of $B_M$.  Since
  the map $x\mapsto\cosh x$ is locally Lipschitz,
  it is Lipschitz on  $[-M,M]$. 
  We denote by $C_M$ a generic constant that
  only depends on $M$. Since, for all $\xi,\eta$
  in $[0,1]$ we have
  $\abs{y(\xi)-y(\eta)}\leq\norm{y_\xi}_{L^1}$, we
  also have
  \begin{align*}
    \abs{\cosh(y(\xi)-y(\eta))-\cosh(\tilde
      y(\xi)-\tilde y(\eta))}&\leq
    C_M\abs{y(\xi)-\tilde y(\xi)-y(\eta)+\tilde
      y(\eta)}\\ &\leq C_M\norm{y-\tilde y}_\Linf.
  \end{align*}
  It follows that, for all $\xi\in[0,1]$,
  \begin{multline*}
    \norm{\cosh(y(\xi)-y(\dott))(U^2y_\xi+\nu)(\dott)-\cosh(\tilde
      y(\xi)-\tilde y(\dott))(\tilde U^2\tilde
      y_\xi+\tilde\nu)(\dott)}_{L^1}\\
    \leq C_M\big(\norm{y-\tilde
      y}_\Linf+\norm{U-\tilde
      U}_\Linf+\norm{y_\xi-\tilde
      y_\xi}_{L^1}+\norm{\nu-\tilde \nu}_{L^1}\big)
  \end{multline*}
  and the map
  $X=(y,U,\nu)\mapsto\frac1{2(e-1)}\int_0^1\cosh(y(\xi)-y(\eta))(U^2y_\xi+\nu)(\eta)\,d\eta$
  which corresponds to the first term in
  \eqref{eq:Psimp} is Lipschitz from $B_M$ to
  $\Linfper$ and the Lipschitz constant only
  depends on $M$. We handle the other terms in
  \eqref{eq:Psimp} in the same way and we prove
  that $\PP$ is Lipschitz from $B_M$ to
  $\Linfper$. Similarly, one proves that $\Q\colon
  B_M\to\Linfper$ is Lipschitz for a Lipschitz
  constant which only depends on $M$. Direct
  differentiation gives the expressions
  \eqref{eq:QPder} for the derivatives $P_\xi$ and
  $Q_\xi$ of $P$ and $Q$. Then, as $\PP$ and $\Q$
  are Lipschitz from $B_M$ to $\Linfper$, we have
  \begin{align*}
    \big\|&\Q(X)_\xi-\Q(\tilde X)_\xi\big\|_{L^1}\\
    &=\norm{y_\xi\PP(X)-\tilde y_\xi\PP(\tilde
      X)-\frac12(U^2y_\xi-\tilde U^2\tilde
      y_\xi)-\nu+\tilde \nu}_{L^1}\\
    &\leq C_M\big(\norm{\PP(X)-\PP(\tilde
      X)}_\Linf+\norm{U-\tilde
      U}_\Linf+\norm{y_\xi-\tilde
      y_\xi}_{L^1}+\norm{\nu-\tilde
      \nu}_{L^1}\big)\\
    &\leq C_M\norm{X-\tilde X}_E.
  \end{align*}
  Hence, we have proved that $\Q\colon
  B_M\to\Wperun$ is Lipschitz for a Lipschitz
  constant that only depends on $M$. We prove the
  corresponding result for $\PP$ in the same way.
\end{proof}

The short-time existence follows from Lemma
\ref{lem:PQ} and a contraction argument. Global
existence is obtained only for initial data which
belong to the set $\F$ as defined below.

\begin{definition}
  \label{def:F}
  The set $\F$ is composed of all $(y,U,\nu)\in E$
  such that
  \begin{subequations}
    \label{eq:lagcoord}
    \begin{align}
      \label{eq:lagcoord1}
      &y\in V_1,\ (y,U)\in
      W^{1,\infty}_{\rm loc}(\Real)\times W^{1,\infty}_{\rm loc}(\Real) ,\ \nu\in L^\infty,\\
      \label{eq:lagcoord2}
      &y_\xi\geq0,\ \nu\geq0,\ y_\xi+\nu\geq c\text{
        almost everywhere, for some constant $c>0$},\\
      \label{eq:lagcoord3}
      &y_\xi \nu=y_\xi^2U^2+U_\xi^2\text{ almost everywhere}.
    \end{align}
  \end{subequations}
\end{definition}

\begin{lemma}
  The set $\F$ is preserved by the equation
  \eqref{eq:sys}, that is, if 
  $X(t)$ solves \eqref{eq:sys} for $t\in[0,T]$
  with initial data $X_0\in\F$, then $X(t)\in\F$ for
  all $t\in[0,T]$.
\end{lemma}
The proof is basically the same as in
\cite{HolRay:07a}.
\begin{theorem}
  \label{th:global}
  For any $\bar X=(\bar y,\bar U,\bar \nu)\in\F$,
  the system \eqref{eq:sys} admits a unique global
  solution $X(t)=(y(t),U(t),\nu(t))$ in
  $C^1(\Real_+,E)$ with initial data $\bar X=(\bar
  y,\bar U,\bar\nu)$. We have $X(t)\in\F$ for all
  times. Let the mapping
  $S\colon\F\times\Real_+\to\F$ be defined as
  \begin{equation*}
    S_t(X)=X(t).
  \end{equation*}
  Given $M>0$ and $T>0$, we define $B_M$ as before,
  that is,
  \begin{equation}
    \label{eq:defBM2}  
    B_M=\{X=(y,U,\nu)\in E\mid 
    \norm{U}_{\Wperun}+\norm{y_\xi}_{L^1}+\norm{\nu}_{L^1}\leq
    M\}.
  \end{equation}
  Then there exists a constant $C_M$ which depends
  only on $M$ and $T$ such that, for any two
  elements $X_\alpha$ and $X_\beta$ in $B_M$, we
  have
  \begin{equation}
    \label{eq:stabSt}
    \norm{S_tX_\alpha-S_tX_\beta}_E\leq C_M\norm{X_\alpha-X_\beta}_E
  \end{equation}
  for any $t\in[0,T]$.
\end{theorem}

\begin{proof}
  By using Lemma \ref{lem:PQ}, we proceed using a
  contraction argument and obtain the existence of
  short time solutions to \eqref{eq:sys}. Let
  $T$ by the maximal time of existence and assume
  $T<\infty$. Let $(y,U,\nu)$ be a solution of
  \eqref{eq:sys} in $C^1([0,T),E)$ with initial
  data $(y_0,U_0,\nu_0)$. We want to prove that
  \begin{equation}
    \label{eq:enormbound}
    \sup_{t\in[0,T)}\norm{(y(t,\dott),U(t,\dott),\nu(t,\dott))}_E<\infty.
  \end{equation}
  From \eqref{eq:sys}, we get 
  \begin{align}
    \notag
    \int_0^1 \nu(t,\xi)\,d\xi&=\int_0^1 \nu(0,\xi)\,d\xi+\int_0^1\int_0^t(-2Q\,Uy_\xi+\left(3U^2-2P\right)U_\xi)(t,\xi)\,dtd\xi\\
    \notag
    &=\int_0^1 \nu(0,\xi)\,d\xi+\int_0^t\int_0^1(U^3-2PU)_\xi(t,\xi)\,d\xi dt\\
    \label{eq:presen}
    &=\int_0^1 \nu(0,\xi)\,d\xi.
  \end{align}
  Hence,
  $\norm{\nu(t,\dott)}_{L^1}=\norm{\nu(0,\dott)}_{L^1}$. This
  identity corresponds to the conservation of the
  total energy. We now consider a fixed time
  $t\in[0,T)$  which we omit
  in the notation when there is no
  ambiguity. For $\xi$ and $\eta$ in $[0,1]$, we
  have $\abs{y(\xi)-y(\eta)}\leq1$ because $y$ is
  increasing and $y(1)-y(0)=1$. From
  \eqref{eq:lagcoord3}, we infer $U^2y_\xi\leq
  \nu$ and, from \eqref{eq:Qsimp}, we obtain
  \begin{equation*}
    \abs{Q}\leq\frac1{e-1}\int_0^1\sinh(\vert y(\xi)-y(\eta)\vert )\nu(\eta)\,d\eta
    +\int_0^1e^{-\abs{y(\xi)-y(\eta)}}\nu(\eta)\,d\eta.
  \end{equation*}
  Hence, 
  \begin{equation}
    \label{eq:boundQ}
    \norm{Q(t,\dott)}_{L^\infty}\leq C\norm{\nu(t,\dott)}_{L^1}=C\norm{\nu(0,\dott)}_{L^1}
  \end{equation}
  for some constant $C$. Similarly, one prove that
  $\norm{P(t,\dott)}_{L^\infty}\leq C
  \norm{\nu(0,\dott)}_{L^1}$ and therefore
  $\sup_{t\in[0,T)}\norm{Q(t,\dott)}_{L^\infty}$ and
  $\sup_{t\in[0,T)}\norm{P(t,\dott)}_{L^\infty}$ are
  finite. Since $U_t=-Q$, it follows that
  \begin{equation}
    \label{eq:boundU}
    \norm{U(t,\dott)}_{L^\infty}\leq\norm{U(0,\dott)}_{L^\infty}+CT\norm{\nu(0,\dott)}_{L^1}
  \end{equation}
  and
  $\sup_{t\in[0,T)}\norm{U(t,\dott)}_{L^\infty}<\infty$. Since
  $y_t=U$, we have that
  $\sup_{t\in[0,T)}\norm{y(t,\dott)}_{L^\infty}$ is also
  finite.  Thus, we have proved that
  \begin{equation*}
    C_1=\sup_{t\in[0,T)}\{\norm{U(t,\dott)}_\Linf+\norm{P(t,\dott)}_\Linf+\norm{Q(t,\dott)}_\Linf\}
  \end{equation*}
  is finite and depends only on $T$ and
  $\norm{U(0,\dott)}_{L^\infty}+\norm{\nu(0,\dott)}_{L^1}$. Let
  $Z(t)=\norm{y_\xi(t,\dott)}_{L^1}+\norm{U_\xi(t,\dott)}_{L^1}+\norm{\nu(t,\dott)}_{L^1}$. Using
  the semi-linearity of \eqref{eq:equivsysdev}
  with respect to $(y_\xi,U_\xi,\nu)$, we obtain
  \begin{equation*}
    Z(t)\leq Z(0)+C\int_0^tZ(\tau)\,d\tau
  \end{equation*}
  where $C$ is a constant depending only on
  $C_1$. It follows from Gronwall's lemma that
  $\sup_{t\in[0,T)}Z(t)$ is finite, and this
  concludes the proof of the global existence. 
  
  Moreover
  we have proved that
  \begin{equation}
    \label{eq:bdnoUX}
    \norm{U(t,\dott)}_{\Wperun}+\norm{y_\xi(t,\dott)}_{L^1}+\norm{\nu(t,\dott)}_{L^1}\leq C_M
  \end{equation}
  for a constant $C_M$ which depends only on $T$
  and
  $\norm{U(0,\dott)}_{\Wperun}+\norm{y_\xi(0,\dott)}_{L^1}+\norm{\nu(0,\dott)}_{L^1}$. Let
  us prove \eqref{eq:stabSt}. Given $T$ and
  $X_\alpha,X_\beta\in B_M$, from Lemma
  \ref{lem:PQ} and \eqref{eq:bdnoUX}, we get that
  \begin{equation*}
    \norm{U_\alpha(t,\dott)-U_\beta(t,\dott)}_{L^\infty}+\norm{Q_\alpha(t,\dott)-Q_\beta(t,\dott)}_{L^\infty}\leq
    C_M\norm{X_\alpha(t)-X_\beta(t)}_E
  \end{equation*}
  where $C_M$ is a generic constant which depends only on
  $M$ and $T$. Using again \eqref{eq:equivsysdev} and
  Lemma \ref{lem:PQ}, we get that for a given time
  $t\in[0,T]$,
  \begin{multline*}
    \norm{U_{\alpha\xi}-U_{\beta\xi}}_{L^1}
    +\norm{\frac{1}{2}\nu_\alpha+\left(\frac{1}{2}U_\alpha^2-P_\alpha\right)y_{\alpha\xi}-\frac{1}{2}\nu_\beta-\left(\frac{1}{2}U_\beta^2-P_\beta\right)y_{\beta\xi}}_{L^1}\\
    +\norm{-2Q_\alpha\,U_\alpha
      y_{\alpha\xi}+\left(3U_\alpha^2-2P_\alpha\right)U_{\alpha\xi}+2Q_\beta\,U_\beta
      y_{\beta\xi}-\left(3U_\beta^2-2P_\beta\right)U_{\beta\xi}}_{L^1}\\
    \leq C_M\norm{X_\alpha-X_\beta}_{E}.
  \end{multline*}
  Hence,
  $\norm{F(X_\alpha(t))-F(X_\beta(t))}_{E}\leq
  C_M\norm{X_\alpha(t)-X_\beta(t)}_E$ where $F$ is
  defined as in \eqref{eq:condequivsys}. Then,
  \eqref{eq:stabSt} follows from Gronwall's lemma
  applied to \eqref{eq:condequivsys}.
\end{proof}

\section{Relabeling invariance}

We denote by $\Gr$ the subgroup of the group of
homeomorphisms on the unit interval defined as
follows: 
\begin{definition}\label{def:group}
  \label{def:G} Let $\Gr$ be the set of all
  functions $f$ such that $f$ is invertible,
  \begin{align}
    \label{eq:Gcond1}
    &f\in W_{\rm loc}^{1,\infty}(\Real),\
    f(\xi+1)=f(\xi)+1 \text{ for all }\xi\in\Real\text{, and }\\
    \label{eq:Gcond2}
    &f-\id\text{ and }f^{-1}-\id\text{ both belong to }\Wper.
  \end{align}
\end{definition}
The set $\Gr$ can be
interpreted as the set of relabeling
functions. Note that $f\in\Gr$ implies that
\begin{equation*}
  \frac{1}{1+\alpha}\leq f_\xi\leq 1+\alpha
\end{equation*}
for some constant $\alpha>0$. This condition is
also almost sufficient as Lemma 3.2 in
\cite{HolRay:07a} shows. Given a triplet
$(y,U,\nu)\in\F$, we denote by $h$ the total
energy $\norm{\nu}_{L^1}$.  We define the subsets
$\F_\alpha$ of $\F$ as follows
\begin{equation*}
  \F_\alpha=\{X=(y,U,\nu)\in\F\mid  \frac{1}{1+\alpha}\leq \frac1{1+h}(y_\xi+\nu)\leq 1+\alpha\}.
\end{equation*}
The set $\F_0$ is then given by
\begin{equation}
  \label{eq:defF0}
  \F_0=\{X=(y,U,\nu)\in\F\mid  y_\xi+\nu=1+h\}.
\end{equation}
We have $\F=\cup_{\alpha\geq0}\F_\alpha$. We
define the action of the group $\Gr$ on $\F$.
\begin{definition}
  We define the map $\Phi\colon\Gr\times\F\to\F$
  as follows
  \begin{equation*}
    \left\{
      \begin{aligned}
        \bar y&=y\circ f,\\
        \bar U&=U\circ f,\\
        \bar\nu&=\nu\circ f f_\xi,
      \end{aligned}
    \right.
  \end{equation*}
  where $(\bar y,\bar U,\bar
  \nu)=\Phi(f,(y,U,\nu))$. We denote $(\bar y,\bar
  U,\bar\nu)=(y,U,\nu)\act f$.
\end{definition}
\begin{proposition}\label{prop:action}
  The map $\Phi$ defines a group action of $\Gr$ on $\F$.
\end{proposition}

\begin{proof}
  By the definition it is clear that $\Phi$ satisfies the fundamental property of a group action, that is 
  $X\act f_1 \act f_2=X\act (f_1\circ f_2)$ for all $X\in\F$ and $f_1$, $f_2\in\Gr$. It remains to prove that $X\act f$ indeed belongs to $\F$. We denote $\hat X=(\hat y,\hat U,\hat \nu)=X\act f$, then it is not hard to check that 
  $\hat y(\xi+1)=\hat y(\xi)+1$, $\hat U(\xi+1)=\hat U(\xi)$, and $\hat\nu(\xi+1)=\hat\nu(\xi)$ for all $\xi\in\Real$.
  By definition we have $\hat v=v\circ f f_\xi$, and we will now prove that 
  \begin{equation*}
    \hat y_\xi=y_\xi \circ f f_\xi,\quad  \text{ and }\quad \hat U_\xi=U_\xi \circ f f_\xi,
  \end{equation*}
  almost everywhere. Let $B_1$ be the set where $y$ is differentiable and $B_2$ the set where $\hat y$ and $f$ are differentiable. Using Rademacher's theorem, we get that $\meas(B_1^c)=\meas(B_2^c)=0$. For $\xi\in B_3=B_2\cap f^{-1}(B_1)$, we consider a sequence $\xi_i$ converging to $\xi$ with $\xi_i\not =\xi$ for all $i\in\N$. We have 
  \begin{equation}\label{diffhaty}
    \frac{y(f(\xi_i))-y(f(\xi))}{f(\xi_i)-f(\xi)}\frac{f(\xi_i)-f(\xi)}{\xi_i-\xi}=\frac{\hat y(\xi_i)-\hat y(\xi)}{\xi_i-\xi}.
  \end{equation}
  Since $f$ is continuous, $f(\xi_i)$ converges to $f(\xi)$ and, as $y$ is differentiable at $f(\xi)$, the left-hand side of 
  \eqref{diffhaty} tends to $y_\xi\circ f(\xi)f_\xi(\xi)$, the right-hand side of \eqref{diffhaty} tends to $\hat y_\xi(\xi)$, and we get 
  \begin{equation}\label{diffhaty2}
    y_\xi(f(\xi))f_\xi(\xi)=\hat y_\xi(\xi),
  \end{equation}
  for all $\xi\in B_3$. Since $f^{-1}$ is Lipschitz
  continuous, one-to-one, and $\meas (B_1^c)=0$, we
  have $\meas(f^{-1}(B_1)^c)=0$ and therefore
  \eqref{diffhaty2} holds almost everywhere.  
  One proves the
  other identity similarly. As $f_\xi>0$ almost
  everywhere, we obtain immediately that
  \eqref{eq:lagcoord2} and \eqref{eq:lagcoord3} are
  fulfilled. That \eqref{eq:lagcoord1} is also
  satisfied follows from the following
  considerations: $\norm{\hat
    y_\xi}_{L^1}=\norm{y_\xi}_{L^1}$, as $y_\xi$ is
  periodic with period $1$. The same argument
  applies when considering $\norm{\hat U_\xi}_{L^1}$
  and $\norm{\hat \nu}_{L^1}$. As $U$ is periodic
  with period $1$, we can also conclude that
  $\norm{\hat U}_{L^\infty}=\norm{U}_{L^\infty}$. As
  $f\in\Gr$, one obtains that $\norm{\hat
    y}_{L^\infty}$ is bounded, but not equal to 
  $\norm{y}_{L^\infty}$.
\end{proof}
Note that the set $B_M$ is invariant with respect
to relabeling while the $E$-norm is not, as the
following example shows: Consider the function
$y(\xi)=\xi\in V_1$, and $f\in\Gr$, then this
yields
\begin{equation*}
  \norm{y(f(\xi))}_{L^\infty([0,1])}=\norm{f(\xi)}_{L^\infty([0,1])}.
\end{equation*}
Hence, the $L^\infty$-norm of $y(f(\xi))$ will
always depend on $f$.

Since $\Gr$ is acting on $\F$, we can consider the
quotient space $\quot$ of $\F$ with respect to the
group action.  Let us introduce the subset $\H$ of
$\F_0$ defined as follows
\begin{equation}\label{eq:calH}
  \H=\{(y,U,\nu)\in\F_0\mid  \int_0^1y(\xi)\,d\xi=0\}.
\end{equation}
It turns out that $\H$ contains a unique
representative in $\F$ for each element of $\quot$,
that is, there exists a bijection between $\H$ and
$\quot$. In order to prove this we introduce two
maps $\Pi_1\colon\F\to\F_0$ and
$\Pi_2\colon\F_0\to\H$ defined as follows
\begin{equation}
  \label{eq:defpi1}
  \Pi_1(X)=X\act f^{-1}
\end{equation}
with
$f=\frac1{1+h}(y+\int_{0}^\xi\nu(\eta)\,d\eta)\in\Gr$
and $X=(y,U,\nu)$, and
\begin{equation}
  \label{eq:defpi2}
  \Pi_2(X)=X(\xi-a)
\end{equation}
with $a=\int_0^1y(\xi)\,d\xi$. First, we have to
prove that $f$ indeed belongs to $\Gr$. We have
\begin{align*}
  f(\xi+1)&=\frac{1}{1+h}\big(y(\xi+1)+\int_0^{\xi+1}\nu(\eta)\,d\eta\big)\\
  &=\frac{1}{1+h}\big(y(\xi)+1+\int_0^\xi\nu(\eta)\,d\eta+h\big)=f(\xi)+1
\end{align*}
and this proves \eqref{eq:Gcond1}. Since
$(y,U,\nu)\in\F$, there exists a constant $c\geq1$ such
that $\frac1c\leq f_\xi\leq c$ for almost every
$\xi$ and therefore \eqref{eq:Gcond2} follows from
an application of Lemma 3.2 in
\cite{HolRay:07a}. After noting that the group
action lets  the quantity
$h=\norm{\nu}_{L^1}$ invariant, it is not hard to check that
$\Pi_1(X)$ indeed belongs to $\F_0$, that is,
$\frac{1}{1+\bar h}(\bar y_\xi+\bar\nu)=1$ where
we denote $(\bar y,\bar U,\bar\nu)=\Pi_1(X)$. Let
us prove that $(\bar y,\bar
U,\bar\nu)=\Pi_2(y,U,\nu)$ belongs to $\H$ for any
$(y,U,\nu)\in\F_0$. On the one hand, we have
$\frac{1}{1+\bar h}(\bar y_\xi+\bar\nu)=1$ because
$\bar h=h$ and $\frac1{1+h}(y_\xi+\nu)=1$ as
$(y,U,\nu)\in\F_0$. On the other hand,
\begin{equation}
  \label{eq:compint1}
  \int_0^1\bar
  y(\xi)\,d\xi=\int_{-a}^{1-a}y(\xi)\,d\xi=\int_0^1
  y(\xi)\,d\xi+\int_{-a}^0y(\xi)\,d\xi+\int_1^{1-a}y(\xi)\,d\xi
\end{equation}
and, since $y(\xi+1)=y(\xi)+1$, we obtain
\begin{equation}
  \label{eq:compint2}
  \int_0^1\bar y(\xi)\,d\xi=\int_0^1
  y(\xi)\,d\xi+\int_{-a}^0y(\xi)\,d\xi+\int_0^{-a}y(\xi)\,d\xi-a=\int_0^1y(\xi)\,dx-a=0.
\end{equation}
Thus $\Pi_2(X)\in\H$. Note that the definition
\eqref{eq:defpi2} of $\Pi_2$ can be rewritten as
\begin{equation*}
  \Pi_2(X)=X\act \tau_a
\end{equation*}
where $\tau_a:\xi\mapsto\xi-a$ denotes the
translation of length $a$ so that $\Pi_2(X)$ is a
relabeling  of $X$.

\begin{definition}
  We denote by $\Pi$ the projection of $\F$ into
  $\H$ given by $\Pi_1\circ\Pi_2$.
\end{definition}
One checks directly that $\Pi\circ\Pi=\Pi$. The
element $\Pi(X)$ is the unique relabeled version
of $X$ which belongs to $\H$ and therefore we have the following result.
\begin{lemma}\label{lemma:bijection}
  The sets $\quot$ and $\H$ are in bijection.
\end{lemma}
Given any element $[X]\in\quot$, we associate
$\Pi(X)\in\H$. This mapping is well-defined and is
a bijection.
\begin{lemma}
  \label{lem:equivPi}
  The mapping $S_t$ is equivariant, that is, 
  \begin{equation}\label{equivS}
    S_t(X\act f)=S_t(X)\act f.
  \end{equation}
\end{lemma}
\begin{proof}
  For any $X_0=(y_0,U_0,\nu_0)\in\F$ and $f\in
  \Gr$, we denote $\bar X_0=(\bar y_0,\bar U_0,
  \bar \nu_0)=X_0\act f$, $X(t)=S_t(X_0)$, and
  $\bar X(t)=S_t(\bar X_0)$.  We claim that
  $X(t)\act f$ satisfies \eqref{eq:sys} and
  therefore, since $X(t)\act f$ and $\bar X(t)$
  satisfy the same system of differential
  equations with the same initial data, they are
  equal. We denote $\hat X(t)=(\hat y(t), \hat
  U(t), \hat \nu(t))=X(t)\act f$. Then we obtain
  \begin{equation*}
    \hat U_t=\frac{1}{4}\int_\Real \sign(\xi-\eta)\exp\big(-\sign(\xi-\eta)(\hat y(\xi)-y(\eta))\big)\big[U^2y_\xi+\nu\big](\eta)d\eta.
  \end{equation*}
  As $\hat y_\xi(\xi)=y_\xi(f(\xi))f_\xi(\xi)$ and $\hat \nu(\xi)=\nu(f(\xi))f_\xi(\xi)$ for almost every $\xi\in\Real$, we obtain after the change of variables $\eta=f(\eta^\prime)$,
  \begin{equation*}
    \hat U_t=\frac{1}{4}\int_\Real\sign(\xi-\eta)\exp\big(-\sign(\xi-\eta)(\hat y(\xi)-\hat y(\eta))\big)\big[\hat U^2 \hat y_\xi +\hat \nu\big](\eta)d\eta.
  \end{equation*}
  Treating similarly the other terms in \eqref{eq:sys}, it follows that $(\hat y, \hat U, \hat \nu)$ is a solution of \eqref{eq:sys}. Thus, since $(\hat y, \hat U, \hat \nu)$ and $(\bar y, \bar U,\bar \nu)$ satisfy the same system of ordinary differential equations with the same initial conditions, they are equal and \eqref{equivS} is proved.
\end{proof}

From this lemma we get that
\begin{equation}
  \label{eq:PiSt}
  \Pi\circ S_t\circ \Pi=\Pi\circ S_t.
\end{equation}

\begin{definition} \label{def:barS}
  We define the semigroup $\bar S_t$ on $\H$ as
  \begin{equation*}
    \bar S_t=\Pi\circ S_t.
  \end{equation*}
\end{definition}

The semigroup property of $\bar S_t$ follows from
\eqref{eq:PiSt}. Using the same approach as in
\cite{HolRay:07a}, we can prove that $\bar S_t$ is
continuous with respect to the norm of $E$. It
follows basically of the continuity of the mapping
$\Pi$ but $\Pi$ is not Lipschitz continuous and
the goal of the next section is to improve this
result and find a metric that makes $\bar S_t$
Lipschitz continuous.

\section{Lipschitz metric for the semigroup $\bar S_t$}

\begin{definition}
  Let $X_\alpha,X_\beta\in\F$, we define
  $J(X_\alpha,X_\beta)$ as
  \begin{equation}
    \label{eq:defJ}
    J(X_\alpha,X_\beta)=\inf_{f,g\in\Gr}\norm{X_\alpha\act f-X_\beta\act g}_E.
  \end{equation}
\end{definition}

Note that, for any $X_\alpha,X_\beta\in\F$ and
$f,g\in\Gr$, we have
\begin{equation}
  \label{eq:invralJ}
  J(X_\alpha\act f,X_\beta\act g)=J(X_\alpha,X_\beta).
\end{equation}
It means that $J$ is invariant with respect to
relabeling. The mapping $J$ does not satisfy the
triangle inequality, which is the reason why we introduce the
mapping $d$.

\begin{definition}\label{def:alberto}
  Let $X_\alpha,X_\beta\in\F$, we define
  $d(X_\alpha,X_\beta)$ as
  \begin{equation}
    \label{eq:defdist}
    d(X_\alpha,X_\beta)=\inf \sum_{i=1}^NJ(X_{n-1},X_n)
  \end{equation}
  where the infimum is taken over all finite sequences
  $\{X_n\}_{n=0}^N\in\F$ which satisfy
  $X_0=X_\alpha$ and $X_N=X_\beta$.
\end{definition}

For any $X_\alpha,X_\beta\in\F$ and $f,g\in\Gr$,
we have
\begin{equation}
  \label{eq:invralJ_d}
  d(X_\alpha\act f,X_\beta\act g)=d(X_\alpha,X_\beta),
\end{equation}
and $d$ is also invariant with respect to
relabeling.

\begin{remark} The definition of the metric
  $d(X_\alpha,X_\beta)$ is the discrete version of
  the one introduced in \cite{BHR}. In \cite{BHR},
  the authors introduce the metric that we
  denote here  as $\tilde d$ where
  \begin{equation*}
    \tilde d(X_\alpha,X_\beta)=\inf\int_{0}^1\tnorm{X_s(s)}_{X(s)}\,ds
  \end{equation*}
  where the infimum is taken over all smooth path
  $X(s)$ such that $X(0)=X_\alpha$ and
  $X(1)=X_\beta$ and the triple norm of an element
  $V$ is defined at a point $X$ as
  \begin{equation*}
    \tnorm{V}=\inf_{g}\norm{V-gX_\xi}
  \end{equation*}
  where $g$ is a scalar function, see \cite{BHR}
  for more details.  The metric $\tilde d$ also
  enjoys the invariance relabeling property
  \eqref{eq:invralJ_d}. The idea behind the
  construction of $d$ and $\tilde d$ is the same:
  We measure the distance between two points
  in a way where two
  relabeled versions of the same point are
  identified. The difference is that in the case
  of $d$ we use a set of points whereas in the
  case of $\tilde d$ we use a curve to join two
  elements $X_\alpha$ and $X_\beta$. Formally, we
  have
  \begin{equation}
    \label{eq:disconteq}
    \lim_{\delta\to 0}\frac1{\delta}J(X(s),X(s+\delta))=\tnorm{X_s}_{X(s)}.
  \end{equation}
\end{remark}

We need to introduce the subsets of bounded energy
in $\F_0$.

\begin{definition}
  We denote by $\F^M$ the set
  \begin{equation*}
    \F^M=\{X=(y,U,\nu)\in \F\mid  h=\norm{\nu}_{L^1}\leq M\}
  \end{equation*}
  and let $\H^M=\H\cap\F^M$.
\end{definition}
The important propery of the set $\F^M$ is that it
is preserved both by the flow, see
\eqref{eq:presen}, and relabeling. Let us prove
that
\begin{equation}
  \label{eq:incFMBM}
  B_{M}\cap\H\subset \H^M \subset B_{\bar M}\cap\H
\end{equation}
for $\bar M=6(1+M)$ so that the sets $B_M\cap\H$
and $\H^M$ are in this sense equivalent. From
\eqref{eq:defF0}, we get
$\norm{y_\xi}_{L^\infty}\leq 1+M$ which implies
$\norm{y_\xi}_{L^1}\leq 1+M$. By
\eqref{eq:lagcoord3}, we get that $U_\xi^2\leq
y_\xi\nu\leq\frac12(y_\xi^2+\nu^2)\leq
\frac12(y_\xi+\nu)^2\leq \frac12(1+h)^2$ and
therefore $\norm{U_\xi}_{L^1}\leq 1+M$. Since
$\int_{0}^1y_\xi(\eta)\,d\eta=1$ and $y_\xi\geq0$,
the set $\{\xi\in[0,1]\mid 
y_\xi(\xi)\geq\frac12\}$ has strictly positive
measure. For a point $\xi_0$ in this set, we get,
by \eqref{eq:lagcoord3}, that
$U^2(\xi_0)\leq\frac{\nu(\xi_0)}{y_\xi(\xi_0)}\leq
2(1+M)$. Hence, $\norm{U}_{L^\infty}\leq
\abs{U(\xi_0)}+\norm{U_\xi}_{L^1}\leq 3(1+M)$ and,
finally,
\begin{equation*}
  \norm{U}_{\Wperun}+\norm{y_\xi}_{L^1}+\norm{\nu}_{L^1}\leq 6(1+M),
\end{equation*}
which concludes the proof of \eqref{eq:incFMBM}.

\begin{definition}
  Let $d_M$ be the metric on $\H^M$ which is
  defined, for any $X_\alpha,X_\beta\in\H^M$, as
  \begin{equation}
    \label{eq:defdM}
    d_M(X_\alpha,X_\beta)=\inf \sum_{i=1}^NJ(X_{n-1},X_n)
  \end{equation}
  where the infimum is taken over all finite
  sequences $\{X_n\}_{n=0}^N\in\H^M$ which
  satisfy $X_0=X_\alpha$ and $X_N=X_\beta$.
\end{definition}

\begin{lemma}
  \label{lem:LinfbdJ}
  For any $X_\alpha,X_\beta\in\H^M$, we have
  \begin{equation}
    \label{eq:LinfbdJ}
    \norm{y_\alpha-y_\beta}_{L^\infty}+\norm{U_\alpha-U_\beta}_{L^\infty}+\abs{h_\alpha-h_\beta}\leq C_M d_M(X_\alpha,X_\beta)
  \end{equation}
  for some fixed constant $C_M$ which depends only
  on $M$.
\end{lemma}
\begin{proof}
  First, we prove that for any
  $X_\alpha,X_\beta\in\H^M$, we have
  \begin{equation}
    \label{eq:linfcompj}
    \norm{y_\alpha-y_\beta}_{L^\infty}+\norm{U_\alpha-U_\beta}_{L^\infty}+\abs{h_\alpha-h_\beta}\leq C_M J(X_\alpha,X_\beta)
  \end{equation}
  for some constant $C_M$ which depends only on
  $M$.  By a change of variables in the integrals,
  we obtain
  \begin{align*}
    \abs{h_\alpha-h_\beta}&=\abs{\int_{0}^1\nu_\alpha\circ
      f f_\xi\,d\xi-\int_{0}^1\nu_\beta\circ g
      g_\xi\,d\xi}\\
    &\leq\norm{X_\alpha\act f-X_\beta\act g}_E.
  \end{align*}
  We have 
  \begin{align}
    \notag
    \norm{y_\alpha-y_\beta}_{L^\infty}&+\norm{U_\alpha-U_\beta}_{L^\infty}\\
    \notag &\leq \norm{X_\alpha\act
      f-X_\beta\act g}_E+\norm{y_\beta\circ
      f-y_\beta\circ
      g}_{L^\infty}+\norm{U_\beta\circ
      f-U_\beta\circ g}_{L^\infty}\\
    \label{eq:linfdifb}
    &\leq \norm{X_\alpha\act f-X_\beta\act
      g}_E+(\norm{y_{\beta\xi}}_{L^\infty}+\norm{U_{\beta\xi}}_{L^\infty})\norm{f-g}_{L^\infty}.
  \end{align}
  From the definition of $\H^M$ we get that,
  for any element $X=(y,U,\nu)\in\H^M$, we have
  $\norm{y_{\xi}}_{L^\infty}+\norm{\nu}_{L^\infty}\leq
  2(1+M)$. Since $U_\xi^2\leq y_\xi\nu$, from
  \eqref{eq:lagcoord3}, it follows that
  $\norm{U_\xi}_{L^\infty}\leq2(1+M)$. Thus,
  \eqref{eq:linfdifb} yields
  \begin{equation}
    \label{eq:linfdifb2}
    \norm{y_\alpha-y_\beta}_{L^\infty}+\norm{U_\alpha-U_\beta}_{L^\infty}\leq \norm{X_\alpha\act f-X_\beta\act
      g}_E+4(1+M)\norm{f-g}_{L^\infty}.
  \end{equation}
  We denote by $C_M$ a generic constant which
  depends only on $M$. The identity
  \eqref{eq:linfcompj} will be proved when we
  prove
  \begin{equation}
    \label{eq:bfdifffg}
    \norm{f-g}_{L^\infty}\leq C_M\norm{X_\alpha\act f-X_\beta\act g}_E.
  \end{equation}
  By using the definition of $\H$, we get that
  \begin{align}
    \notag
    \norm{f_\xi-g_\xi}_{L^1}&=\norm{\frac{1}{1+h_\alpha}(y_{\alpha\xi}\circ
      f+\nu_{\alpha}\circ
      f)f_\xi-\frac{1}{1+h_\beta}(y_{\beta\xi}\circ
      g+\nu_{\beta}\circ g)g_\xi}_{L^1}\\
    \notag
    &\leq\frac{\abs{h_\alpha-h_\beta}}{1+h_\beta}+\frac{1}{1+h_\beta}\norm{X_\alpha\act
      f-X_\beta\act g}_E\\
    \label{eq:difgxibd}
    &\leq C_M\norm{X_\alpha\act f-X_\beta\act
      g}_E.
  \end{align}
  Let $\delta=g(0)-f(0)$.  Similar to
  \eqref{eq:compint1} and \eqref{eq:compint2}, we
  can conclude that
  \begin{align}\nn
    \int_0^1 y_\beta\circ (f+\delta)f_\xi d\xi
    & = \int_{f(0)+\delta}^{f(0)+1+\delta} y_\beta d\xi \\ \nn
    & =\int_{f(0)+\delta}^0 y_\beta d\xi +\int_0^1 y_\beta d\xi +\int_1^{1+f(0)+\delta} y_\beta d\xi \\ \nn
    & = \int_{f(0)+\delta}^0 y_\beta d\xi +\int_0^1 y_\beta d\xi +\int_0^{f(0)+\delta} y_\beta d\xi +f(0)+\delta \\ \nn
    & =f(0)+\delta.
  \end{align}
  Thus we have $\delta=\int_0^1y_\beta\circ
  (f+\delta) f_\xi\,d\xi-f(0)$ and analogously
  $0=\int_0^1 y_\beta \circ (f)f_\xi d\xi-f(0)$.
  Hence,
  \begin{equation}
    \label{eq:deltabd0}
    \abs{\delta}=\abs{\int_0^1y_\beta\circ (f+\delta)
      f_\xi\,d\xi-\int_0^1y_\alpha\circ ff_\xi\,d\xi}.
  \end{equation}
  By \eqref{eq:difgxibd}, we get that
  \begin{equation}
    \label{eq:gmfdebd}
    \norm{g-f-\delta}_{L^\infty}\leq
    \norm{f_\xi-g_\xi}_{L^1}\leq 
    C_M\norm{X_\alpha\act f-X_\beta\act g}_E.
  \end{equation}
  Then, since
  \begin{align*}
    \norm{y_\beta\circ(f+\delta)-y_\beta\circ
      g}_{L^\infty}&\leq\norm{y_{\beta\xi}}_{L^\infty}\norm{f+\delta-g}_{L^\infty}\\
    &\leq C_M\norm{X_\alpha\act f-X_\beta\act
      g}_E,
  \end{align*}
  we obtain that
  \begin{align}
    \notag
    \norm{y_\alpha\circ f-y_\beta\circ(f+\delta)}_{L^\infty}&\leq  \norm{y_\alpha\circ f-y_\beta\circ g}_{L^\infty}+ \norm{y_\beta\circ g-y_\beta\circ(f+\delta)}_{L^\infty}\\
    \label{eq:difyfdelt}
    &\leq C_M\norm{X_\alpha\act f-X_\beta\act
      g}_E.
  \end{align}
  Then, \eqref{eq:deltabd0} yields
  \begin{equation}
    \label{eq:deltabd}
    \abs{\delta}\leq C_M\norm{X_\alpha\act f-X_\beta\act
      g}_E.
  \end{equation}
  From \eqref{eq:gmfdebd} and \eqref{eq:deltabd},
  \eqref{eq:bfdifffg} and therefore
  \eqref{eq:linfcompj} follows. For any $\epsi>0$,
  we consider a sequence $\{X_n\}_{n=0}^N$ in
  $\H^M$ such that $X_0=X_\alpha$ and
  $X_N=X_\beta$ and
  $\sum_{i=1}^NJ(X_{n-1},X_n)\leq
  d_M(X_\alpha,X_\beta)+\epsi$.  We have
  \begin{align*}
    \norm{y_\alpha-y_\beta}_{L^\infty}+&\norm{U_\alpha-U_\beta}_{L^\infty}+\abs{h_\alpha-h_\beta}\\
    &\leq
    \sum_{n=1}^{N}\norm{y_{n-1}-y_n}_{L^\infty}+\norm{U_{n-1}-U_n}_{L^\infty}+\abs{h_{n-1}-h_n}\\
    &\leq C_M\sum_{n=1}^{N}J(X_{n-1},X_n)\\
    &\leq C_M(d_M(X_\alpha,X_\beta)+\epsi).
  \end{align*}
  Since $\epsi$ is arbitrary, we get
  \eqref{eq:LinfbdJ}.
\end{proof}

From the definition of $d$, we obtain that
\begin{equation}
  \label{eq:dequiv}
  d(X_\alpha,X_\beta)\leq\norm{X_\alpha-X_\beta}_E,
\end{equation}
so that the metric $d$ is weaker than the
$E$-norm.

\begin{lemma} \label{lemma:distance} The mapping
  $d_M:\H^M\times\H^M\to\Real_+$ is a metric on
  $\H^M$.
\end{lemma}
\begin{proof}
  The symmetry is embedded in the definition of
  $J$ while the construction of $d_M$ from $J$ takes
  care of the triangle inequality. From Lemma
  \ref{lem:LinfbdJ}, we get that
  $d_M(X_\alpha,X_\beta)=0$ implies that
  $y_\alpha=y_\beta$, $U_\alpha=U_\beta$ and
  $h_\alpha=h_\beta$. Then, the definition
  \eqref{eq:defF0} of $\F_0$ implies that
  $\nu_\alpha=\nu_\beta$.
\end{proof}

\begin{remark}
  \label{rem:compmetric}
  In \cite{HolRay:07a}, a metric on $\H$ is
  obtained simply by taking the norm of $E$. The
  authors prove that the semigroup is continuous
  with respect to this norm, that is, given a
  sequence $X_n$ and $X$ in $\H$ such that
  $\lim_{n\to\infty} \norm{X_n-X}_E$, we have
  $\lim_{n\to\infty}\norm{\bar S_tX_n-\bar
    S_tX}_E=0$. However, $\bar S_t$ is not
  Lipschitz in this norm. From \eqref{eq:dequiv},
  we see that the distance introduced in
  \cite{HolRay:07a} is stronger than the one
  introduced here. (The definition of $E$ in
  \cite{HolRay:07a} differs slightly from the one
  employed here, but the statements in this remark
  remain valid).
\end{remark}

We can now prove the Lipschitz stability theorem
for $\bar S_t$.
\begin{theorem}
  \label{th:stab} Given $T>0$ and $M>0$, there
  exists a constant $C_M$ which depends only on $M$
  and $T$ such that, for any
  $X_\alpha,X_\beta\in\H^M$ and $t\in[0,T]$, we
  have
  \begin{equation}
    \label{eq:stab}
    d_M(\bar S_tX_\alpha,\bar S_tX_\beta)\leq C_Md_M(X_\alpha,X_\beta).
  \end{equation}
\end{theorem}

\begin{proof}
  By the definition of $d_M$, for any $\epsi>0$,
  there exists a sequences $\{X_n\}_{n=0}^N$ in
  $\H^M$ and functions $\{f_n\}_{n=1}^{N-1}$,
  $\{g_n\}_{n=1}^{N-1}$ in $\Gr$ such that
  $X_0=X_\alpha$, $X_N=X_\beta$ and
  \begin{equation}
    \label{eq:sumXnm1}
    \sum_{i=1}^N\norm{X_{n-1}\act f_{n-1}-X_{n}\act g_{n-1}}_E\leq
    d_M(X_\alpha,X_\beta)+\epsi.
  \end{equation}
  Since $\H^M\subset B_{\bar M}$ for $\bar
  M=6(1+M)$, see \eqref{eq:incFMBM}, and $B_{\bar
    M}$ is preserved by relabeling, we have that
  $X_{n}\act f_n$ and $X_{n}\act g_{n-1}$ belong to
  $B_{\bar M}$. From the Lipschitz stability
  result given in \eqref{eq:stabSt}, we obtain
  that
  \begin{equation}
    \label{eq:normSxnm1}
    \norm{S_t(X_{n-1}\act f_{n-1})-S_t(X_{n}\act g_{n-1})}_E\leq C_M\norm{X_{n-1}\act f_{n-1}-X_{n}\act g_{n-1}}_E,
  \end{equation}
  where the constant $C_M$ depends only on $M$ and
  $T$. Introduce
  \begin{equation*}
    \bar X_n=X_n\act f_n,\  \bar X_n^t=S_t(\bar X_n), \text{ for }n=0,\ldots,N-1,
  \end{equation*}
  and
  \begin{equation*}
    \tilde X_n=X_{n}\act g_{n-1},\ \tilde X_n^t=S_t(\tilde X_n), \text{ for }n=1,\ldots,N.
  \end{equation*}
  Then \eqref{eq:sumXnm1} rewrites as 
  \begin{equation}
    \label{eq:sumXnm1b}
    \sum_{i=1}^N\norm{\bar X_{n-1}-\tilde X_{n}}_E\leq
    d_M(X_\alpha,X_\beta)+\epsi
  \end{equation}
  while \eqref{eq:normSxnm1} rewrites as
  \begin{equation}
    \label{eq:normSnnot}
    \norm{\bar X_{n-1}^t-\tilde X_{n}^t}_E\leq C_M\norm{\bar X_{n-1}-\tilde X_{n}}_E.
  \end{equation}
  We have
  \begin{equation*}
    \Pi(\bar X_0^t)=\Pi\circ S_t(X_0\act f_0)=\Pi\circ (S_t(X_0)\act f_0)=\Pi\circ S_t(X_0)=\bar S_t(X_\alpha)
  \end{equation*}
  and similarly $\Pi(\tilde X_N^t)=\bar
  S_t(X_\beta)$. We consider the sequence in
  $\H^M$ which consists of $\{\Pi \bar
  X_n^t\}_{n=0}^{N-1}$ and $\bar
  S_t(X_\beta)$. The set $\F^M$ is preserved by
  the flow and by relabeling. Therefore, $\{\Pi
  \bar X_n^t\}_{n=0}^{N-1}$ and $\bar
  S_t(X_\beta)$ belong to $\H^M$. The endpoints
  are $\bar S_t(X_\alpha)$ and $\bar
  S_t(X_\beta)$. From the definition of the metric
  $d_M$, we get
  \begin{align}
    \notag d_M(\bar S_t(X_\alpha),\bar
    S_t(X_\beta))&\leq\sum_{n=1}^{N-1} \left(J(\Pi
      \bar X_{n-1}^t,\Pi\bar X_n^t)\right)+J(\Pi
    \bar
    X_{N-1}^t,\bar S_t(X_\beta))\\
    \label{eq:dtilS}
    &=\sum_{n=1}^{N-1} \left(J(\bar X_{n-1}^t,\bar
      X_n^t)\right)+J(\bar X_{N-1}^t,\tilde
    X_N^t))&\text{ by \eqref{eq:invralJ}.}
  \end{align}
  By using the equivariance of $S_t$, we obtain
  that
  \begin{equation}
    \label{eq:tilXrelbarX}
    \begin{aligned}
      \tilde X_n^t&=S_t(\tilde X_n)=S_t((\bar X_n\act f_n^{-1})\act g_{n-1})\\
      &=S_t(\bar X_n)\act (f_n^{-1}\circ g_{n-1})=\bar X_n^t\act (f_n^{-1}\circ g_{n-1}).
    \end{aligned}
  \end{equation}
  Hence, by using \eqref{eq:invralJ}, that is, the
  invariance of $J$ with respect to relabeling,
  we get from \eqref{eq:dtilS} that
  \begin{align*}
    d_M(\bar S_t(X_\alpha),\bar
    S_t(X_\beta))&\leq\sum_{n=1}^{N-1}
    \left(J(\bar X_{n-1}^t,\tilde
      X_n^t)\right)+J(\bar
    X_{N-1}^t,\tilde X_N^t)\\
    &\leq \sum_{n=1}^N\norm{\bar X_{n-1}^t-\tilde
      X_n^t}_E &\text{ by \eqref{eq:dequiv} }\\
    &\leq C_M\sum_{n=1}^N\norm{\bar X_{n-1}-\tilde
      X_n}_E &\text{  by \eqref{eq:normSnnot}}\\
    &\leq C_M(d_M(X_\alpha,X_\beta)+\epsi).
  \end{align*}
  After letting $\epsi$ tend to zero, we obtain
  \eqref{eq:stab}.
\end{proof}

\section{From Lagrangian to Eulerian coordinates}

We now introduce a second set of coordinates, the
so--called Eulerian coordinates.  Therefore let us
first consider $X=(y,U,\nu)\in \F$. We can define
the Eulerian coordinates as in \cite{HolRay:07a}
and also obtain the same mappings between Eulerian
and Lagrangian coordinates. For completeness we
will state the results here.

\begin{definition}  \label{def:D}
  The set $\D$ consists of all pairs $(u,\mu)$ such that 
  \begin{enumerate}
  \item 
    $u\in H^1_{\rm per}$, and   
  \item 
    $\mu$ is a positive Radon measure whose absolute continuous part, $\mu_{\rm ac}$, satisfies 
    \begin{equation}
      \mu_{\rm ac}=(u^2+u_x^2)dx.
    \end{equation}
  \end{enumerate}
\end{definition}

We can define a mapping, denoted by $L$, from $\D$
to $\H\subset \F$:
\begin{definition}\label{defL}
  For any $(u,\mu)$ in $\D$, let
  \begin{equation}\label{def:EtoF}
    \begin{aligned} 
      h& = \mu([0,1)), \\
      y(\xi)& = \sup \{y \mid F_\mu(y)+y<(1+h)\xi\},\\
      \nu(\xi)& = (1+h)-y_\xi(\xi),\\
      U(\xi)& =u\circ y(\xi),
    \end{aligned}
  \end{equation}
  where 
  \begin{equation}
    F_\mu(x)=\left\{
      \begin{aligned}
        \mu([\,0,x)) \quad & \text{ if } x>0,\\
        0 \quad & \text{ if }x=0,\\
        -\mu([\,x,0)) \quad & \text{ if } x<0.
      \end{aligned}
    \right .
  \end{equation}
  Then $(y,U,\nu)\in \F_0$. We define
  $L(u,\mu)=\Pi(y,U,\nu)$.
\end{definition}

Thus from any initial data $(u_0,\mu_0)\in \D$, we can construct a solution of \eqref{eq:sys} in $\F$ with initial data $X_0=L(u_0,\mu_0)\in\F$. It remains to go back to the original variables, which is the purpose of the mapping $M$, defined as follows.

\begin{definition}\label{def:FtoE}
  For any $X\in\F$, then $(u,\mu)$ given by
  \begin{equation}
    \label{eq:defL}
    \begin{aligned}
      u(x)=U(\xi) & \text{ for any }   \xi \text{ such that } x=y(\xi),\\
      \mu& =y_\# (\nu d\xi),
    \end{aligned}
  \end{equation}
  belongs to $\D$. We denote by $M$ the mapping
  from $\F$ to $\D$ which for any $X\in\F$
  associates the element $(u,\mu)\in\D$ given by
  \eqref{eq:defL}. 
\end{definition}
The mapping $M$ satisfies
\begin{equation}
  \label{eq:propM}
  M=M\circ \Pi.
\end{equation}
The inverse of $L$ is the restriction of $M$ to
$\H$, that is,
\begin{equation}
  \label{eq:inversma}
  L\circ M=\Pi, \quad \text{ and } \quad M\circ L=\id.
\end{equation}

\begin{figure}[h]
  \centering
  \includegraphics[width=10cm]{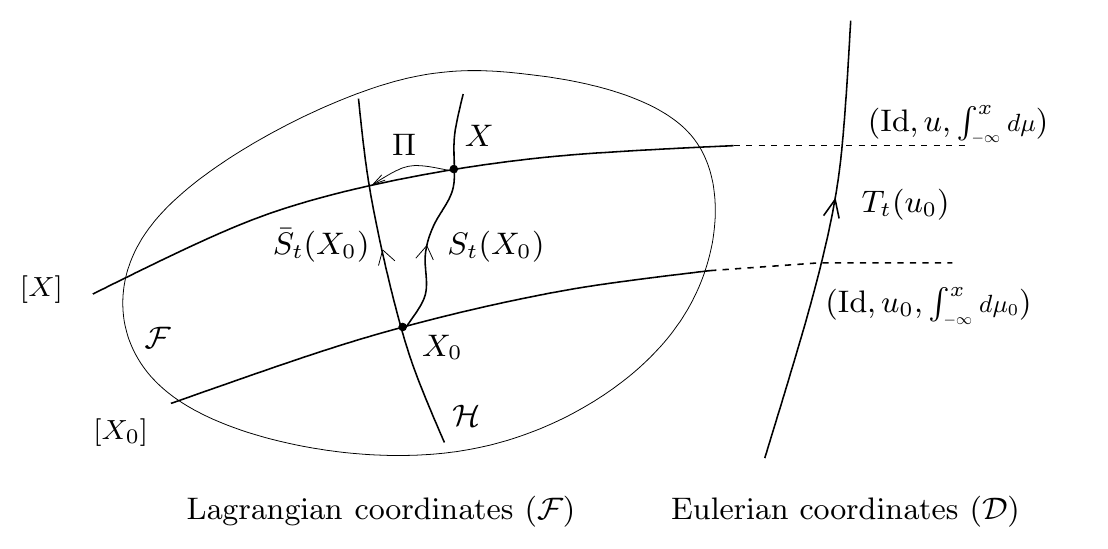}
  \caption{A schematic illustration of  the construction
    of the semigroup. The set $\F$ where the
    Lagrangian variables are defined is
    represented by the interior of the closed
    domain on the left. The equivalence classes $[X]$ and
    $[X_0]$ (with respect to the action of the
    relabeling group $G$) of $X$ and $X_0$,
    respectively, are represented by 
    horizontal curves. To each equivalence class
    there corresponds a unique element in $\H$ and
    $\D$ (the set of Eulerian variables). The sets
    $\H$ and $\D$ are represented by the
    vertical curves.}
  \label{fig:explfig}
\end{figure}

Next we show that we indeed have obtained a solution of the CH equation. 
By a weak solution of the Camassa--Holm equation
we mean the following.
\begin{definition}
  Let $u\colon\Real_+\times\Real \rightarrow \Real$. Assume that $u$ 
  satisfies \\
  (i) $u\in L^\infty([0,\infty), H^1_{\rm per})$, \\
  (ii)  the equations
  \begin{multline}\label{weak1}
    \iint_{\Real_+\times \Real}-u(t,x)\phi_t(t,x)+(u(t,x)u_x(t,x)+P_x(t,x))\phi(t,x)dxdt \\
    = \int_\Real u(0,x)\phi(0,x)dx,
  \end{multline}
  and
  \begin{equation}\label{weak2}
    \iint_{\Real_+\times \Real}(P(t,x)-u^2(t,x)-\frac{1}{2} u_x^2(t,x))\phi(t,x)+P_x(t,x)\phi_x(t,x)dxdt=0,
  \end{equation}
  hold for all $\phi\in C_0^\infty
  ([0,\infty),\Real)$. Then we say that $u$ is a weak
  global solution of the Camassa--Holm equation.
\end{definition}

\begin{theorem}
  Given any initial condition
  $(u_0,\mu_0)\in\D$, we denote
  $(u,\mu)(t)=T_t(u_0,\mu_0)$. Then $u(t,x)$ is a
  weak global solution of the Camassa--Holm
  equation.
\end{theorem}

\begin{proof}
  After making the change of variables
  $x=y(t,\xi)$ we get on the one hand
  \begin{align}\nn
    -\iint_{\Real_+\times\Real} & u(t,x)\phi_t(t,x)dxdt
    = - \iint_{\Real_+\times
      \Real}u(t,y(t,\xi))\phi_t(t,y(t,\xi))y_\xi(t,\xi)d\xi
    dt\\\nn & =- \iint_{\Real_+\times \Real}
    U(t,\xi)[(\phi(t,y(t,\xi))_t-\phi_x(t,y(t,\xi)))y_t(y,\xi)]y_\xi(t,\xi)d\xi
    dt\\\nn
    & = -\iint_{\Real_+\times \Real}[U(t,\xi)y_\xi(t,\xi)(\phi(t,y(t,\xi)))_t-\phi_\xi(t,y(t,\xi))U(t,\xi)^2]d\xi dt\\
    & = \int_\Real
    U(0,\xi)\phi(0,y(0,\xi))y_\xi(0,\xi)d\xi\\\nn
    & \quad + \iint_{\Real_+\times \Real}
    [U_t(t,\xi)y_\xi(t,\xi)+U(t,\xi) y_{\xi
      t}(t,\xi)]\phi(t,y(t,\xi))d\xi dt \\\nn &
    \quad +\iint_{\Real_+\times \Real}
    U^2(t,\xi)\phi_\xi(t,y(t,\xi))d\xi dt\\\nn & =
    \int_\Real u(0,x)\phi(0,x)dx\\\nn & \quad
    -\iint_{\Real_+\times \Real}
    (Q(t,\xi)y_\xi(t,\xi)+U_\xi(t,\xi)U(t,\xi))\phi(t,y(t,\xi))d\xi
    dt,
  \end{align}
  while on the other hand
  \begin{align}\nn
    \iint_{\Real_+\times \Real} & (u(t,x)u_x(t,x)+P_x(t,x))\phi(t,x)dxdt\\
    & = \iint_{\Real_+\times
      \Real}(U(t,\xi)U_\xi(t,\xi)+P_x(t,y(t,\xi))y_\xi(t,\xi))\phi(t,y(t,\xi))d\xi
    dt\\\nn & = \iint_{\Real_+\times \Real}
    (U(t,\xi)U_\xi(t,\xi)+Q(t,\xi)y_\xi(t,\xi))\phi(t,y(t,\xi)) d\xi dt,
  \end{align}
  which shows that \eqref{weak1} is fulfilled. Equation
  \eqref{weak2} can be shown analogously
  \begin{align}\nn
    \iint_{\Real_+\times \Real}&
    P_x(t,x)\phi_x(t,x)dxdt\\\nn
    & = \iint_{\Real_+\times \Real} Q(t,\xi)y_\xi(t,\xi)\phi_x(t, y(t,\xi))d\xi dt\\
    &= \iint_{\Real_+\times \Real}
    Q(t,\xi)\phi_\xi(t,y(t,\xi)) d\xi dt\\\nn & =
    -\iint_{\Real_+\times
      \Real}Q_\xi(t,\xi)\phi(t,y(t,\xi)) d\xi dt\\\nn
    & = \iint_{\Real_+\times \Real} [\frac{1}{2}
    \nu(t,\xi)+(\frac{1}{2}U^2(t,\xi)-P(t,\xi))y_\xi(t,\xi)]\phi(t,y(t,\xi))d\xi
    dt\\\nn &= \iint_{\Real_+\times \Real}[\frac{1}{2}
    u_x^2(t,x)+u^2(t,x)-P(t,x)]\phi(t,x) dx dt.
  \end{align}
  In the last step we used the following
  \begin{align}
    \int_0^1 u^2+u_x^2 dx & =\int_{y(0)}^{y(0)+1} u^2+u_x^2 dx =\int_{y(0)}^{y(1)} u^2+u_x^2 dx \\
    & = \int_{\{\xi\in[0,1] \mid y_\xi (t,\xi)>0\}} U^2y_\xi +\frac{U_\xi^2}{y_\xi}d\xi=\int_0^1 \nu dx,
  \end{align}
 the last equality holds only for almost
    all $t$ because  for almost every
  $t\in \Real_+$ the set $\{\xi\in[0,1] \mid y_\xi
  (t,\xi)>0\}$ is of full measure and therefore
  \begin{equation}
    \label{eq:preseuen}
    \int_0^1 (u^2+ u_x^2) dx=\int_0^1 \nu d\xi=h,
  \end{equation}
  which is bounded by a constant for all times.
  Thus we proved that $u$ is a weak solution of
  the Camassa--Holm equation. 
\end{proof}

Next we return to the construction of the Lipschitz metric on $\D$.
\begin{definition}
  Let
  \begin{equation}
    T_t:=M\bar S_tL \colon \D\rightarrow \D.
  \end{equation}
\end{definition}
Note that, by the definition of $\bar S_t$ and
\eqref{eq:propM}, we also have that
\begin{equation*}
  T_t=M S_tL.
\end{equation*}
Next we show that $T_t$ is a Lipschitz continuous
semigroup by introducing a metric on $\D$. Using
the bijection $L$ transport the topology from $\H$
to $\D$.

\begin{definition} \label{def:dD}
  We define the metric $d_{\D}\colon \D\times \D
  \rightarrow [0,\infty)$ by
  \begin{equation}
    d_{\D}((u,\mu),(\tilde{u},\tilde{\mu}))=d(L(u,\mu),L(\tilde{u},\tilde{\mu})).
  \end{equation}
\end{definition}
The Lipschitz stability of the semigroup $T_t$
follows then naturally from Theorem~\ref{th:stab}. The stability holds on sets of bounded energy
that we now introduce in the following definition. 
\begin{definition}
  Given $M>0$, we define the subsets $\D^M$ of
  $\D$, which corresponds to sets of bounded
  energy, as
  \begin{equation}
    \label{eq:defbdenerD}
    \D^M=\{(u,\mu)\in\D\ | \ \mu([0,1))\leq M\}.
  \end{equation}
  On the set $\D^M$, we define the metric
  $d_{\D^M}$ as
  \begin{equation}
    \label{eq:defdDM}
    d_{\D^M}((u,\mu),(\tilde{u},\tilde{\mu}))=d_M(L(u,\mu),L(\tilde{u},\tilde{\mu}))
  \end{equation}
  where the metric $d_M$ is defined in
  \eqref{eq:defdM}.
\end{definition}
The definition \eqref{eq:defdDM} is well-posed as
we can check from the definition of $L$ that if
$(u,\mu)\in\D^M$ then $L(u,\mu)\in\H^M$. We can
now state our main theorem.
\begin{theorem}\label{th:main}
  The semigroup $(T_t,d_\D)$ is a continuous
  semigroup on $\D$ with respect to the metric
  $d_\D$. The semigroup is Lipschitz continuous on
  sets of bounded energy, that is: Given $M>0$ and
  a time interval $[0,T]$, there exists a constant
  $C$ which only depends on $M$ and $T$ such that,
  for any $(u,\mu)$ and $(\tilde u,\tilde\mu)$ in
  $\D^M$, we have
  \begin{equation*}
    d_{\D^M}(T_t(u,\mu),T_t(\tilde{u},\tilde{\mu}))\leq Cd_{\D^M}((u,\mu),(\tilde{u},\tilde{\mu}))
  \end{equation*}
  for all $t\in[0,T]$.
\end{theorem}

\begin{proof} First, we prove that $T_t$ is a
  semigroup. Since $\bar S_t$ is a mapping from
  $\H$ to $\H$, we have
  \begin{equation*}
    T_{t}T_{t'}=M\bar S_tLM\bar S_{t'}L=M\bar S_t\bar S_{t'}L=M\bar S_{t+t'}L=T_{t+t'}
  \end{equation*}
  where we also use \eqref{eq:inversma} and the
  semigroup property of $\bar S_t$. We now prove
  the Lipschitz continuity of $T_t$. By using
  Theorem \ref{th:stab}, we obtain that
  \begin{align*}
    d_{\D^M}(T_t(u,\mu),T_t(\tilde{u},\tilde{\mu}))
    & = d_M(LM\bar S_tL(u,\mu),
    LM\bar S_tL(\tilde{u},\tilde{\mu}))\\
    & = d_M(\bar S_tL(u,\mu),
    \bar S_tL(\tilde{u},\tilde{\mu}))\\
    &\leq C d_M(L(u,\mu),L(\tilde{u},\tilde{\mu}))\\
    &= C
    d_{\D^M}((u,\mu),(\tilde{u},\tilde{\mu})).
  \end{align*}
\end{proof}

\section{The topology on $\D$}\label{sec:topology}

\begin{proposition} \label{prop:cont1}
  The mapping
  \begin{equation}
    u\mapsto  (u,(u^2+u_x^2) dx)
  \end{equation} 
  is continuous from $H^1_{\rm per}$ into
  $\D$. In other words, given a sequence
  $u_n\in H^1_{\rm per}$ converging to $u\in H^1_{\rm per}$, 
  then $(u_n,(u_n^2+u_{nx}^2)dx)$ converges to $(u,(u^2+u_x^2)
  dx)$ in $\D$.
\end{proposition}

\begin{proof}
  Let $X_n=(y_n, U_n,\nu_n)$ be the image of
  $(u_n, (u_n^2+u_{n,x}^2)dx)$ given as in
  \eqref{def:EtoF} and $X=(y,U,\nu)$ the image of
  $(u,(u^2+u_x^2)dx)$ given as in
  \eqref{def:EtoF}.  We will at first prove that
  $u_n$ converges to $u$ in $H^1_{\rm per}$ implies
  that $X_n$ converges against $X$ in $E$.  Denote
  $g_n=u_n^2+u_{nx}^2$ and $g=u^2+u_x^2$, then
  $g_n$ and $g$ are periodic functions. Moreover,
  as $X_n$, $X\in\F_0$, we have
  $y_{n,\xi}+\nu_n=1+h_n$ and $y_\xi+\nu=1+h$,
  where $h_n=\norm{\nu_n}_{L^1}$ and
  $h=\norm{\nu}_{L^1}$. By Definition~\ref{defL},
  we have that $y_n(0)=0$ and $y(0)=0$, and hence
  \begin{align}
    \int_0^{y_n(\xi)} g_n(x)dx+y_n(\xi)&=\int_0^\xi \nu_n(x)dx +y_n(\xi) =(1+h_n)\xi, \\ \nn 
    \int_0^{y(\xi)} g(x)dx +y(\xi)&=\int_0^\xi \nu(x)dx +y(\xi) =(1+h)\xi.
  \end{align}
  By assumption $u_n \to u$ in $H^1_{per}$, which implies that $u_n\to u$ in $L^\infty$, $g_n\to g$ in $L^1$, and $h_n\to h$. 
  Therefore we also obtain that $y_n\to y$ in $L^\infty$. 
  We have 
  \begin{equation}
    U_n-U=u_n\circ y_n -u\circ y=u_n\circ y_n-u\circ y_n+u\circ y_n-u\circ y.
  \end{equation}
  Then, since $u_n\to u$ in $L^\infty$, also
  $u_n\circ y_n\to u\circ y_n$ in $L^\infty$ and
  as $u$ is in $H^1_{per}$, we also
  obtain that $u\circ y_n\to u\circ y$ in
  $L^\infty$. Hence, it follows that $U_n\to U$ in
  $L^\infty$.  By definition, the measures
  $(u^2+u_x^2)dx$ and $(u_n^2+u_{nx}^2)dx$ have no
  singular part, and we therefore have almost
  everywhere
  \begin{equation}
    y_\xi=\frac{1+h}{1+g\circ y} \quad \text{ and } \quad y_{n\xi}=\frac{1+h_n}{1+g_n\circ y_n}.
  \end{equation}
  Hence 
  \begin{align}\label{approx:zeta1}
    y_\xi-y_{n\xi} & = y_\xi y_{n\xi}\Big(\frac{1+g_n\circ y_n}{1+h_n}-\frac{1+g\circ y}{1+h}\Big)\\ \nn
    & = y_\xi y_{n\xi}\Big(\frac{1+g_n\circ y_n}{1+h_n}-\frac{1+g_n\circ y_n}{1+h} \Big)\\
    & \qquad+ \frac{y_\xi y_{n\xi}}{1+h}(g_n\circ y_n-g\circ y_n+g\circ y_n -g\circ y). \nn
  \end{align}
  In order to show that $\zeta_{n,\xi}\to\zeta_\xi$ in $L^1_{\rm per}$, it suffices to investigate 
  \begin{equation}
    \int_0^1 \vert  g\circ y_n-g\circ y\vert y_\xi y_{n,\xi}d\xi, 
  \end{equation}
  and 
  \begin{equation}
    \int_0^1\vert g_n\circ y_n-g\circ y_n \vert y_\xi y_{n,\xi} d\xi,
  \end{equation}
  as we already know that $h_n\to h$ and therefore $y_{n,\xi}$ and $y_\xi$ are bounded.
  Since $0\leq y_\xi\leq 1+h$, we have
  \begin{equation}\label{approx:zeta2}
    \int_0^1 \vert g\circ y_n-g_n\circ y_n\vert y_\xi y_{n,\xi}d\xi\leq (1+h) \norm{g-g_n}_{L^1}. 
  \end{equation}
  For the second term, let $C=\sup_{n}(1+h_n)\geq
  1$.  Then for any $\varepsilon>0$ there exists a
  continuous function $v$ with compact support
  such that $\norm{g-v}_{L^1}\leq
  \varepsilon/3C^2$ and we can make the following
  decomposition
  \begin{align}\label{approx:zeta3}
    (g\circ y-g\circ y_n)y_{n,\xi}y_\xi&=(g\circ y-v\circ y)y_{n,\xi}y_\xi\\ \nn
    &\quad+(v\circ y-v\circ y_n)y_{n,\xi}y_\xi+(v\circ y_n-g\circ y_n)y_{n,\xi}y_\xi.
  \end{align}
  This implies 
  \begin{equation}
    \int_0^1 \vert g\circ y-v\circ y\vert y_{n,\xi}y_\xi d\xi\leq C\int_0^1 \vert g\circ y-v\circ y\vert y_\xi d\xi\leq \varepsilon/3,  
  \end{equation}
  and analogously we obtain $\int_0^1 \vert g\circ y_n-v\circ y_n\vert y_{n,\xi}y_\xi d\xi\leq \varepsilon/3$. As $y_n\to y$ in $L^\infty$ and $v$ is continuous, we obtain, by applying the Lebesgue dominated convergence theorem, that $v\circ y_n\to v\circ y$ in $L^1$, and we can choose $n$ so big that 
  \begin{equation}
    \int_0^1 \vert v\circ y_n-v\circ y\vert y_{n,\xi}y_\xi d\xi\leq C^2 \norm{v\circ y-v\circ y_n}_{L^1}\leq \varepsilon/3.  
  \end{equation}
  Hence, we showed, that $\int_0^1 \vert g\circ y-g\circ y_n\vert y_{n,\xi}y_\xi d\xi\leq \varepsilon$ and therefore, using \eqref{approx:zeta3},
  \begin{equation}\label{approx:zeta4}
    \lim_{n\to\infty}\int_0^1 \vert g\circ y-g\circ y_n\vert y_{n,\xi}y_\xi d\xi=0.
  \end{equation}
  Combing now \eqref{approx:zeta1}, \eqref{approx:zeta2}, and \eqref{approx:zeta3}, yields $\zeta_{n\xi}\to \zeta_\xi$ in $L^1$, 
  and therefore also $\nu_n\to\nu$ in $L^1$. 
  Because $\zeta_{n,\xi}$ and $\nu_n$ are bounded in $L^\infty$, we also have that $\zeta_{n,\xi}\to\zeta_\xi$ in $L^2$ and $\nu_n\to\nu$ in $L^2$. Since $y_{n,\xi}$, $\nu_n$ and $U_n$ tend to $y_\xi$, $\nu$ and $U$ in $L^2$ and $\norm{U_n}_{L^\infty}$ and $\norm{y_{n,\xi}}_{L^\infty}$, are uniformly bounded, it follows from \eqref{eq:lagcoord3} that 
  \begin{equation}
    \lim_{n\to\infty}\norm{U_{n,\xi}}_{L^2}=\norm{U_\xi}_{L^2}.
  \end{equation}
  Once we have proved that $U_{n,\xi}$ converges weakly to $U_\xi$, this will imply that $U_{n,\xi}\to U_\xi$ in $L^2$.
  For any smooth function $\phi$ with compact support in $[0,1]$ we have 
  \begin{equation}
    \int_\Real U_{n,\xi}\phi d\xi=\int_\Real  u_{n,x}\circ y_ny_{n,\xi}\phi d\xi=\int_\Real u_{n,x}\phi\circ y_n^{-1} d\xi. 
  \end{equation}
  By assumption we have $u_{n,\xi}\to u_\xi$ in $L^2$. Moreover, since $y_n\to y$ in $L^\infty$, the support of $\phi\circ y_n^{-1}$ is contained in some compact set, which can be chosen independently of $n$. Thus, using Lebesgue's dominated convergence theorem, we obtain that $\phi\circ y_n^{-1}\to \phi\circ y^{-1}$ in $L^2$ and therefore 
  \begin{equation}\label{eq:UxiL2}
    \lim_{n\to\infty}\int_\Real U_{n,\xi}\phi d\xi=\int_\Real  u_x\phi\circ y^{-1}d\xi=\int_\Real U_\xi\phi d\xi.
  \end{equation}
  Form \eqref{eq:lagcoord3} we know that $U_{n,\xi}$ is bounded and therefore by a density argument \eqref{eq:UxiL2} holds for any function $\phi$ in $L^2$ and therefore $U_{n,\xi}\to U_\xi$ weakly and hence also in $L^2$. 
  Using now that 
  \begin{equation}
    \norm{U_{n,\xi}-U_\xi}_{L^1}\leq \norm{U_{n,\xi}-U_\xi}_{L^2},
  \end{equation}
  shows that we also have convergence in $L^1$. 
  Thus we obtained
  that $X_n\to X$ in $E$.  As a second and last
  step, we will show that $\Pi_2$ is continuous,
  which then finishes the proof.  We already know
  that $y_n\to y$ in $L^\infty$ and therefore
  $a_n=\int_0^1 y_n(\xi)d\xi$ converges to
  $a=\int_0^1 y(\xi)d\xi$.  Thus we obtain as an
  immediate consequence
  \begin{multline} 
    \norm{U_n(\xi-a_n)-U(\xi-a)}_{L^\infty}\\
    \leq \norm{U_n(\xi-a_n)-U(\xi-a_n)}_{L^\infty}+\norm{U(\xi-a_n)-U(\xi-a)}_{L^\infty},
  \end{multline}
  and hence the same argumentation as before shows that $U_n(\xi-a_n)\to U(\xi-a)$ in $L^\infty$. 
  Moreover, 
  \begin{align}
    \int_0^1 \vert U_{n,\xi}(\xi-a_n)& -U_\xi(\xi-a)\vert d\xi \\ \nn
    & \leq \int_0^1 \vert U_{n,\xi}(\xi-a_n)-U_\xi (\xi-a_n)\vert d\xi +\int_0^1 \vert U_\xi (\xi-a_n)-U_\xi(\xi-a)\vert d\xi\\ \nn
    & \leq \norm{U_{n,\xi}-U_\xi}_{L^1}+\norm{U_\xi(\xi-a_n)-U_\xi(\xi-a)}_{L^1},
  \end{align}
  and again using the same ideas as in the first part of the proof, we have that 
  $U_{n,\xi}(\xi-a_n)\to U_{\xi}(\xi-a)$ in $L^1$, which finally proves the claim, because of \eqref{eq:dequiv}
\end{proof}

\begin{proposition} \label{prop:cont2}
  Let $(u_n, \mu_n)$ be a sequence in
  $\D$ that converges to $(u,\mu)$ in
  $\D$. Then
  \begin{equation}
    u_n\rightarrow u \text{ in } L^\infty_{\rm per} \text{ and } \mu_n \overset{\ast}{\rightharpoonup}\mu.
  \end{equation}
\end{proposition}

\begin{proof}

  Let $X_n=(y_n, U_n, \nu_n)=L(u_n,\mu_n)$
  and $X=(y,U,\nu)=L(u,\mu)$ .  By the
  definition of the metric $d_\D$, we
  have $\lim_{n\to\infty}d(X_n,X)=0$. We immediately obtain 
  that
  \begin{equation}
    X_n \to X \text{ in } L^\infty(\Real),
  \end{equation}
  by Lemma~\ref{lem:LinfbdJ}.  The rest can be
  proved as in \cite[Proposition 5.2]{HolRay:07a}.
\end{proof}

\noindent{\bf Acknowledgments.} 
K. G. gratefully acknowledges the hospitality of
the Department of Mathematical Sciences at the
NTNU, Norway, creating a great working environment
for research during the fall of 2009.

\end{document}